\documentclass[12pt,a4paper]{article}
\usepackage{amsmath,amssymb,amsthm,url,amsfonts,enumerate,graphicx}
\usepackage[a4paper,left=2.7cm,right=2.1cm,top=2cm,bottom=3.2cm]{geometry}

\usepackage[usenames,dvipsnames]{color}

\newcommand\aronly[1]{#1}
\newcommand\jonly[1]{}

\def\Z{{\mathbb Z}} \def\R{{\mathbb R}}  
\long\def\comment#1\endcomment{}

\def\sgn{\mathop{\fam0 sgn}}
\def\con{\mathop{\fam0 Con}}

\def\lk{\mathop{\fam0 lk}}
\def\Int{\mathop{\fam0 Int}}

\def\Cl{\mathop{\fam0 Cl}}
\def\im{\mathop{\fam0 im}}

\renewcommand{\t}[1]{\ensuremath{\widetilde{#1}}}

\theoremstyle{plain}
\newtheorem{theorem}{Theorem}[section]
\newtheorem{lemma}[theorem]{Lemma}

\newtheorem{proposition}[theorem]{Proposition}
\newtheorem{conjecture}[theorem]{Conjecture}

\newtheorem{problem}[theorem]{Problem}

\theoremstyle{definition}
\newtheorem{remark}[theorem]{Remark}

\usepackage{hyperref}
\hypersetup{
   colorlinks   = true, 
   urlcolor     = blue, 
   linkcolor    = blue, 
   citecolor   = red 
}

\begin{document}

\title{Some `converses' to intrinsic linking theorems
\footnote{We would like to thank F. Frick, T. Garaev and anonymous referees for helpful discussions.}}
\author{R. Karasev\footnote{Institute for Information Transmission Problems.
\texttt{http://www.rkarasev.ru/en/about/}.
Email: \texttt{r\_n\_karasev@mail.ru}.
\aronly{Supported by
the Russian Foundation for Basic Research grant 
19-01-00169.}}
 and A. Skopenkov\footnote{Independent University of Moscow.
\texttt{https://users.mccme.ru/skopenko/}.
Email: \texttt{skopenko@mccme.ru}.
\aronly{Supported by the Russian Foundation for Basic Research grant 19-01-00169.}}}
\date{}
\maketitle

\begin{abstract}
A low-dimensional version of our main result is the following `converse' of the Conway--Gordon--Sachs Theorem
on intrinsic linking of the graph $K_6$ in 3-space:

{\it For any integer $z$ there are 6 points $1,2,3,4,5,6$ in 3-space, of which every two $i,j$ are joined by a polygonal line $ij$, the interior of one polygonal line is disjoint with any other polygonal line, the linking number of any pair of disjoint 3-cycles except for $\{123,456\}$ is zero, and for the exceptional pair $\{123,456\}$ is $2z+1$.}

We prove a higher-dimensional analogue, which is a `converse' of a lemma by Segal--Spie\.z.
\end{abstract}

\noindent
{\em MSC 2010}: 57Q35, 57K45, 55S91, 68U05.

\noindent
{\em Keywords:} intrinsic linking, linking number, embedding, almost embedding, deleted product.


\section{Introduction and main result}\label{s:mr}

We start with a low-dimensional intrinsic linking result (Theorem \ref{t:cgs}), its higher-dimensional generalization (Lemma \ref{l:ramsey}), and a low-dimensional analogue (Proposition \ref{p:rams}) of our main result (Theorem \ref{t:ss1ae}).

Disjoint closed polygonal lines $L_1,L_2$ in $\R^3$ (or, more generally, disjoint self-intersecting $k$-sphere and $\ell$-sphere in $\R^{k+\ell+1}$)
are {\it linked modulo 2} if a general position singular cone over $L_1$ intersects $L_2$ at an odd number of points \cite[\S77]{ST80}, \cite[\S4]{Sk}.

\begin{theorem}[Conway--Gordon--Sachs; \cite{CG83, Sa81}]\label{t:cgs}
For any piecewise linear (PL) embedding  $K_6\to \R^3$ there are two disjoint cycles in $K_6$ whose images are  linked modulo 2.
\end{theorem}

For
a survey on `intrinsic linking' results see e.g. \cite{Sk14} and the references therein.

The {\it linking number} $\lk\in\Z$ of disjoint oriented closed polygonal lines in $\R^3$ (or, more generally, of disjoint oriented self-intersecting $k$-sphere and $\ell$-sphere in $\R^{k+\ell+1}$),
is defined in \cite[\S77]{ST80}, \cite[\S4]{Sk}.
For non-oriented closed polygonal lines (or singular spheres) the absolute value $|\lk|$ is well-defined.

This paper is motivated by finding a gap \cite[\S3]{Sk20e} in the proof that embeddability is undecidable in codimension $>1$ \cite{FWZ}.
Theorem \ref{t:cgs} and its higher-dimensional generalization (Lemma \ref{l:ramsey}) give cycles (or spheres) linked modulo 2, i.e., having odd linking number.
The gap was in trying to improve those results to get linking number $\pm1$, not just odd.
Our main result (Theorem \ref{t:ss1ae}) shows that this is not possible.


The following `converse' of Theorem \ref{t:cgs} shows that the existence of two cycles with odd linking number cannot be replaced by the existence of two cycles with $\pm1$ linking number.

\begin{proposition}[proved in \S\ref{s:pr}]\label{p:rams}
For any integer $z\ge0$ there is a PL embedding $K_6\to\R^3$ such that

$\bullet$ the image of any 3-cycle is unknotted,

$\bullet$ for any disjoint 3-cycles in $K_6$ except one pair
the linking number of their images is zero, and

$\bullet$ for the exceptional pair of disjoint 3-cycles we have $|\lk|=2z+1$.
\end{proposition}


A {\bf complex} is a collection of closed  simplices (=faces) of some simplex.
(We abbreviate `finite simplicial complex' to `complex'.)
A {\bf $k$-complex} is a complex containing at most $k$-dimensional simplices.
The {\it body} (or geometric realization) $|K|$ of a complex $K$ is the union of simplices of $K$.
Thus continuous or piecewise-linear (PL) maps $|K|\to\R^d$ and continuous maps $|K|\to S^m$ are defined.
Below we abbreviate $|K|$ to $K$; no confusion should arise.

A map $g:K\to\R^d$ of a complex $K$ is called an {\bf almost embedding} if $g\alpha\cap g\beta=\emptyset$ for any two disjoint simplices $\alpha,\beta\subset K$.

\begin{lemma}[{\cite[Lemma 1.4]{SS92}}]\label{l:ramsey} For any integers $0\le\ell\le k$ there is a complex
$F_-$ of dimension $\max\{k,\ell+1\}$ containing disjoint subcomplexes $\Sigma^k\cong S^k$ and
$\Sigma^\ell\cong S^\ell$, PL embeddable into $\R^{k+\ell+1}$ and such that for any PL almost embedding $f:F_-\to\R^{k+\ell+1}$ the images $f\Sigma^k$ and $f\Sigma^\ell$ are linked modulo 2.
\end{lemma}



Let us define $F_-,\Sigma^k$, and $\Sigma^\ell$ of Lemma \ref{l:ramsey}.
For this, define a complex $F=F_{k,\ell}$ (this is $P(k,\ell)$ of \cite{SS92}).
Let $[n]:=\{1,2,\ldots,n\}$.
The vertex set is $[k+\ell+3]\cup\{0\}$.
The simplices are formed by all the simplices of dimension at most $k$ of $[k+\ell+3]$, and all the simplices of dimension at most $\ell+1$ that contain $0$.
In other words,
$$F_{k,\ell} := \left(\ [k+\ell+3]\cup\{0\}\ ,
\ {[k+\ell+3]\choose\le k+1}\cup\left\{\{0\}\cup\sigma\ :\ \sigma\in{[k+\ell+3]\choose\le \ell+1}\right\}\ \right).$$
Here ${[n]\choose \le m}$ is the set of all subsets of $[n]$ having at most $m$ elements.

\smallskip
{\it Comment.} Observe that $F_{1,0}$ is the non-planar graph $K_5$.
More generally, $F_{k,k-1}$ is the $k$-skeleton of the $(2k+2)$-simplex, which is not embeddable into $\R^{2k}$.
We have
$F_{k,k}=\con F_{k,k-1}$.
The complex $F_{k,\ell}$ is not embeddable into $\R^{k+\ell+1}$ for $0\le\ell\le k$ by Lemma \ref{l:ramsey}.
(see the following definition of $\Sigma^\ell$, $\Sigma^k$ and $F_-$).

\smallskip
Let $\Delta^{\ell+1}\subset F$ be the $(\ell+1)$-simplex with the vertex set $\{0,1,2,\ldots,\ell+1\}$ and $\Sigma^\ell=\partial\Delta^{\ell+1}$.
Let $\Sigma^k\subset F$ be the boundary sphere of the $(k+1)$-simplex with the vertex set $\{\ell+2,\ell+3,\ldots,k+\ell+3\}$.
Finally, define
$$F_-=F_{k,\ell,-}:=F-\Int\Delta^{\ell+1}.$$
Our main result (Theorem \ref{t:ss1ae}) shows that in Lemma \ref{l:ramsey} the linking number being odd cannot be replaced by the linking number being $\pm1$.

For a PL almost embedding $f:F_-\to\R^{k+\ell+1}$ we have $f\Sigma^k\cap f\Sigma^\ell=\emptyset$, so denote
$$|\lk f|:=|\lk(f\Sigma^k,f\Sigma^\ell)|\in\Z.$$

\begin{theorem}[proved in \S\ref{s:pr}]\label{t:ss1ae}
For any integers $1\le\ell\le k$ and $z\ge0$ there is a PL almost embedding $f:F_-\to\R^{k+\ell+1}$ such that
$|\lk f|=2z+1$.
\end{theorem}

\subsection*{Open problems}


The paper \cite{Ga22} announces that the analogue for $(k,\ell)=(1,0)$ of Theorem \ref{t:ss1ae} is wrong.
It would be interesting to know if the analogue for $\ell=0$, $k\ge2$ of Theorem \ref{t:ss1ae} holds.


It would be interesting to know if

$\bullet$ for any graph $K$ there is a PL embedding $K\to\R^3$ such that for any disjoint cycles in $K$
the linking number of their images is different from $\pm1$.

$\bullet$ for any integers $0<\ell<k$ and a $k$-complex $K$ there is a PL embedding (or almost embedding) $K\to\R^{k+\ell+1}$ such that for any disjoint $k$-sphere and $\ell$-sphere in $K$
the linking number of their images is different from $\pm1$.


$\bullet$ for any integers $1<\ell\le k$ and a $k$-complex $K$ there is a PL embedding (or almost embedding)
$K\to\R^{k+\ell}$ such that for any disjoint $k$-simplex and $\ell$-simplex in $K$ the intersection of their images is not transversal.

A negative answer would be a natural integer-valued generalization of Theorem \ref{t:cgs} and Lemma \ref{l:ramsey}.


It would be interesting to know if `almost' can be deleted from Theorem \ref{t:ss1ae}.
For $k<2\ell$ this follows by Theorem \ref{t:ss1ae} and Proposition \ref{p:delp}.a.
For $k\ge2\ell$ this can perhaps be proved analogously to Theorem \ref{t:ss1ae}, using Lemma \ref{l:ss}.b and the method of \cite[\S2]{Sk98}, see the survey \cite[\S10]{RS99}.

\begin{problem}\label{t:ss1}
Take fixed $d,k$ such that $8\le d\le \frac{3k+1}2$.
Is there an algorithm recognizing PL almost embeddability of $k$-complexes in $\R^d$?
\end{problem}

Theorem \ref{t:ss1ae} allows to deduce a negative answer to Problem \ref{t:ss1} (cf. \cite{FWZ}, \cite[\S3]{Sk20e}) for the `extreme' case $2d=3k+1=6\ell+4$, $\ell$ even,
from \cite[Conjecture 4.8.b]{Sk20e}, see details in \cite[end of \S4]{Sk20e}.
Analogously one deduces such an answer
from \cite[Conjecture 4.8.a]{Sk20e}.
Analogously one deduces undecidability for embeddings
from a version of Theorem \ref{t:ss1ae} without `almost' and \cite[Conjecture 4.8.a]{Sk20e}.

\subsection*{More information on intrinsic linking}\label{s:intli}

\begin{remark}\label{r:ssl} (a) Lemma \ref{l:ramsey} is an important step in the proof of the important results \cite{SS92, SSS, MTW, ST17} on incompleteness of the deleted product criterion for embeddability of $k$-complexes in $\R^d$ for $2d<3k+3$ (Proposition \ref{p:delp}.a), and on NP-hardness of recognition of (almost) embeddability of $k$-complexes in $\R^d$ for $2d<3k+3$.

(b) The proof of Lemma \ref{l:ramsey} in \cite[\S1]{SS92} uses the cohomological Smith index;
a simpler argument by application of \cite[Lemma 6]{ST17} is presented after Theorem \ref{t:ramse}.
\jonly{See also \cite[Remark 1.6.c,d]{KS20}.}

\aronly{(c) Our statement of Lemma \ref{l:ramsey} does not coincide with \cite[Lemma 1.4]{SS92}.
So observe that the condition \cite[1.4.a]{SS92} can be achieved by general position, and the condition $l<k$ of \cite[Lemma 1.4]{SS92} can be replaced by $l\le k\ge2$ without changing the proof.
Indeed, the case $k=\ell=0$ is clear and the case $k-1=\ell=0$ follows by Theorem \ref{t:cgs}.

(d) The following phrase in \cite[p. 278]{SS92} requires an explanation:
`Since the map $p|_{\widetilde\Delta^{l+1}}$ induces an isomorphism
$H_{l+1} (\widetilde\Delta^{l+1}, \partial\widetilde\Delta^{l+1}) \to H_{l+1} (\Delta^{l+1}, \partial\Delta^{l+1})$
of the ordinary homology groups, the map $p^*$ induces an epimorphism
$H_{n+1} (\widetilde P^*, \widetilde T) \to H_{n+1} (P^*,T) $
of the equivariant homology groups.'
(In  \cite[p. 278]{SS92} $T^*$ is a typo and should be replaced by $T$.)}
\end{remark}


Theorem \ref{t:ss1ae} can be regarded as a `converse' also to the following `intrinsic linking' result, Theorem \ref{t:ramse}.
This result is a non-trivial generalization of Theorem \ref{t:cgs}, of Remark \ref{r:sgs}.a, and of their analogues for $k$-complexes in $\R^{2k+1}$, as well as a simple generalization of Lemma \ref{l:ramsey}   and of \cite[Lemma 5]{ST17}.

\begin{theorem}\label{t:ramse} For any integers $0\le\ell\le k$ there is a $k$-complex $F'=F'_{k,\ell}$ containing disjoint subcomplexes $\Sigma^k_j\cong S^k$ and $\Sigma^\ell_j\cong S^\ell$, $j\in{[k+\ell+3]\choose k+2}$,
PL embeddable into $\R^{k+\ell+1}$ and such that for any PL almost embedding $f:F'\to\R^{k+\ell+1}$
the number of linked modulo 2 unordered pairs of the images $f\Sigma^k_j$ and $f\Sigma^\ell_j$ is odd.
\end{theorem}

The above-mentioned analogues for $k$-complexes in $\R^{2k+1}$ are obtained by taking in Theorem \ref{t:ramse} $k=\ell$ and $F'$ the $k$-skeleton of the $(k+l+3)$-simplex.
They are proved in \cite{LS98, Ta00}, see survey \cite[\S4]{Sk16}.
The index argument of \cite[\S1]{SS92} (see Remark \ref{r:ssl}.c) has a simple generalization to Theorem \ref{t:ramse}; thus the analogues are implicit in \cite{SS92}.

\begin{proof}[Sketch of a proof of Theorem \ref{t:ramse}]
Take $F'$ to be the complex whose vertex set is $[k+\ell+3]\cup\{0\}$, and whose simplices are formed by all the simplices of dimension at most $k$ of $[k+\ell+3]$, and all the simplices of dimension at most $\ell$ that contain $0$.
For $j\in{[k+\ell+3]\choose k+2}$ let

$\bullet$ $\Sigma^k_j\subset F'$ be the boundary sphere of the $(k+1)$-simplex with the vertex set $j$.

$\bullet$ $\Sigma^\ell_j\subset F'$ be the boundary sphere of the $(\ell+1)$-simplex with the vertex set
$\{0\}\cup([k+\ell+3]-j)$.

Now the proof is analogous to the following proof of Lemma \ref{l:ramsey}, and so to \cite[\S3]{ST17}.
\end{proof}

\begin{proof}[Proof of Lemma \ref{l:ramsey}]
For a complex $K$, a general position PL map $f\colon K\to \R^d$ and $\dim K<d$ define {\it the van Kampen number} $v(f)\in\Z_2$ to be the parity of the number of points $x\in \R^d$ such that $x\in f(\sigma) \cap f(\tau)$ for some disjoint simplices $\sigma, \tau \in  K$ with $\dim \sigma + \dim \tau = d$.
The lemma follows because
{\it $v(f)=1$ for any general position PL map $f\colon F\to \R^{k+\ell+1}$.}
For {\it some} $f$ this is Lemma \ref{l:ss}.a below.
Then for {\it any} $f$ this holds by the following \cite[Lemma 6]{ST17} (verification of its assumptions is analogous to \cite[Lemma 7]{ST17}): {\it Let $d$ be an integer and $K$ a finite complex such that for every pair $\sigma, \tau$ of disjoint $s$- and $t$-simplices in $K$ with $s+t=d-1$ the following two numbers have the same parity:

$\bullet$ the number of $(s+1)$-simplices $\nu$ containing $\sigma$ and disjoint with $\tau$;

$\bullet$ the number of $(t+1)$-simplices $\mu$ containing $\tau$ and disjoint with $\sigma$.

Then $v(f)$ is independent of a general position PL map $f\colon |K|\to \R^d$.}
\end{proof}

\section{Proofs of Proposition \ref{p:rams} and Theorem \ref{t:ss1ae}}\label{s:pr}

For any disjoint oriented cycles $\sigma,\tau$ in $K_6$, put $\lk_f(\sigma,\tau):=\lk(f\sigma,f\tau)\in\Z$.
Observe that $\lk_f(\sigma,\tau)=\lk_f(\tau,\sigma)$, so assume that the argument of $\lk_f$ is an unordered pair.

For an oriented edge $c$ of $K_6$ issuing out of vertex $A$ and going to vertex $B$, and a vertex $C\not\in c$ denote by $cC$ the oriented cycle $CA\cup c\cup BC$ in $K_6$.

\begin{lemma}\label{l:kar} Let  $a,b$ be disjoint oriented edges of $K_6$ and $f:K_6\to\R^3$ a PL embedding such that any 3-cycle in $f(K_6)$ is unknotted.
Then there is a PL embedding $g:K_6\to\R^3$ such that any 3-cycle in $g(K_6)$ is unknotted,
for the remaining vertices $P,Q$ of $K_6$ we have
$$\lk\phantom{}_f(aP,bQ)-\lk\phantom{}_g(aP,bQ)=\lk\phantom{}_f(aQ,bP)-\lk\phantom{}_g(aQ,bP)=+1$$
and $\lk_f(\sigma,\tau)=\lk_g(\sigma,\tau)$ for any other unordered pair $\sigma,\tau$.
\end{lemma}

\begin{proof}
Informally, we obtain $g$ by turning $f(a)$ around $f(b)$ once.
Let us present an accurate construction.
Take a point $O\in\R^3$ and general position arcs $Of(V)$ joining $O$ to the images of the vertices of $K_6$.
For an oriented edge $c$ of $K_6$ issuing out of vertex $A$ and going to vertex $B$ denote by $Of(c)$ the oriented cycle $Of(A)\cup f(c)\cup f(B)O$.
Take an embedded oriented 2-disk $\delta\subset\R^3$ such that $\delta\cap f(K_6)$ is the union of

$\bullet$ a point $f(b)\cap\delta\subset\Int\delta$ of sign $+1$, and

$\bullet$ an arc  $f(a)\cap\partial\delta$ at which the orientations from $f(a)$ and from $\partial\delta$ are the opposite.

Define $g:K_6\to\R^3$ by `pushing a finger along $\delta$', i.e., so that

$\bullet$ $g(f^{-1}(\partial\delta))=\Cl(\partial\delta-f(a))$,

$\bullet$ $g=f$ outside $f^{-1}(\partial\delta)$,

$\bullet$ $\lk(Of(a),Of(b))-\lk(Og(a),Of(b))=+1$, and

$\bullet$ $\lk(Of(a),Of(c))=\lk(Og(a),Of(c))$ for any oriented edge $c\not\in\{a,b\}$.

Since any 3-cycle in $f(K_6)$ is unknotted, $f(K_6)\cap\delta\subset f(a\cup b)$ and no 3-cycle in $K_6$ containing $a$ contains $b$, any 3-cycle in $g(K_6)$ is unknotted.

Observe that $\lk_f(\sigma,\tau)$ equals to the sum of 9 summands of the form $\lk(Of(d),Of(e))$, where $d$ and $e$ are oriented edges of $\sigma$ and $\tau$.
Hence
$$\lk\phantom{}_f(aP,bQ)-\lk\phantom{}_g(aP,bQ)=\lk(Of(a),Of(b))-\lk(Og(a),Of(b))=\delta\cap f(b)=+1.$$
The relation $\lk_f(aQ,bP)-\lk_g(aQ,bP)=+1$ follows by exchanging $P$ and $Q$.
Analogously $\lk_f(\sigma,\tau)=\lk_g(\sigma,\tau)$ for any other unordered pair $\{\sigma,\tau\}$.
\end{proof}

\begin{proof}[Proof of Proposition \ref{p:rams}]
Denote the vertices of $K_6$ by $1,2,\ldots,6$.
It is known that
{\it there is a PL embedding $f:K_6\to\R^3$ such that any 3-cycle in $f(K_6)$ is unknotted, $\lk_f(123,456)=+1$
and $\lk_f(\sigma,\tau)=0$ for any other unordered pair $\sigma,\tau$ of disjoint oriented cycles in $K_6$.}

Make the modification of Lemma \ref{l:kar} for $(aP,bQ)=(123,456),(162,435),(234,561)$.
We have $\lk_f(ijk,pqr)=\lk_f(jki,pqr)=-\lk_f(jik,pqr)$ whenever $[6]=\{i,j,k,p,q,r\}$.
Hence the resulting change of the symmetric matrix $\lk_f$ is
$$\left(\{123,456\}+\{126,453\}\right)+\left(\{162,345\}+\{165,342\}\right)+\left(\{561,234\}+\{564,231\}\right)=$$
$$=2\{123,456\}.$$
Thus making the same modification $z$ times we obtain the required PL embedding.
\end{proof}

\begin{remark}\label{r:sgs}
(a) Theorem \ref{t:cgs} was proved in the following stronger form:

{\it For any piecewise linear (PL) embedding  $K_6\to \R^3$ the number of linked modulo 2 unordered pairs of
images of two disjoint cycles in $K_6$ is odd.}


(b) Part (a) is implied by the following assertion \jonly{(see proof in \cite[Remark 2.2.b]{KS20})} and the known fact stated at the beginning of the proof of Proposition \ref{p:rams} (this is essentially the standard argument).

{\it For any two PL embeddings $f,g:K_6\to\R^3$ the symmetric matrix $\lk_f$ can be obtained from
the symmetric matrix $\lk_g$ by several transformations described in Lemma \ref{l:kar}.}

\aronly{{\it Proof.} We may assume that $g=f$ outside the interior of an edge $a$.
Then analogously to the calculations in the proof of Lemma \ref{l:kar} we see that $\lk_f$ is obtained from
$\lk_g$ by the transformations described in Lemma \ref{l:kar} for all the 6 edges of $K_6$ disjoint from $a$. \qed}

(c) Proposition \ref{p:rams} shows that there are no linear relations or congruences on numbers $\lk_f(\sigma,\tau)$ except (a).
\aronly{
The following is a combinatorial illustration of this fact.

{\it There is no map $\xi:X\to\{+1,-1\}$ from the set $X$ of unordered pairs of disjoint oriented cycles $\sigma,\tau$ of length 3 in $[6]$ such that
$\xi(ijk,pqr)=-\xi(ijr,pqk)$ whenever $[6]=\{i,j,k,p,q,r\}$.}

This follows because otherwise
$$\xi(123,456)=\xi(231,564)=-\xi(234,561)=-\xi(342,615)=$$
$$=\xi(345,612)=\xi(126,453)=-\xi(123,456).$$}
\end{remark}

\begin{proof}[Proof of Theorem \ref{t:ss1ae}]\footnote{We are grateful to F. Frick for allowing us to present this proof based on an idea he suggested. Alternative (earlier) proofs are presented in \S\ref{s:alt}\aronly{-\S\ref{s:delpro}.} \jonly{and in \cite[\S4, \S5]{KS20}.}}
If a closed polygonal line $L\subset\R^3$ is unknotted then the fundamental group of the complement is $\Z$.
Hence a closed polygonal line in $\R^3-L$ is null-homotopic if and only if it is null-homologous, i.e., if and only if it has zero linking number with $L$.

So if two closed polygonal lines in $\R^3$ have zero linking number and the second of them is unknotted then the first of them spans a mapped 2-disk disjoint from the second one.
Hence the inductive base $k=\ell=1$ follows by Proposition \ref{p:rams}.

Let us prove the inductive step.
If $k>1$, then either $k>\ell$ or $\ell>1$.

If $k>\ell$, observe that
$$F_{k,\ell}=F_{k-1,\ell}\cup\con\left(F_{k-1,\ell}\cap F_{k,\ell-1}\right)\cup {k+\ell+2\choose k+1},$$
where the vertex of the cone is $k+\ell+3$.
The same formula is correct with $F_{k,\ell},F_{k-1,\ell}$ replaced by $F_{k,\ell,-},F_{k-1,\ell,-}$.
Since $k>\ell$, by the inductive hypothesis there is a PL almost embedding $f:F_{k-1,\ell,-}\to\R^{k+\ell}$ such that $|\lk f|=2z+1$.
Extend it to a map $f':F_{k,\ell,-}\to\R^{k+\ell+1}$ as follows.
Extend $f$ conically over the cone, with the vertex in the upper half-space of $\R^{k+\ell+1}$ w.r.t. $\R^{k+\ell}$.
Map the $k$-faces of ${k+\ell+2\choose k+1}$ to the lower half-space of $\R^{k+\ell+1}$ w.r.t. $\R^{k+\ell}$.
Since $k>\ell$, we have $2(k+1)>k+\ell+2$, so any two such $k$-faces intersect.
Thus the extension $f'$ is a PL almost embedding.
Since $f'\Sigma^k$ is the `suspension' over $f\Sigma^{k-1}$ and $f'=f$ on $\Sigma^\ell$,
we have $|\lk f'|=|\lk f|=2z+1$.

If $\ell>1$, observe that
$$F_{k,\ell}=F_{k,\ell-1}\cup\con\left(F_{k,\ell-1}\cap F_{k-1,\ell}\right)\cup 0*{k+\ell+2\choose \ell+1},$$
where the vertex of the cone is $k+\ell+3$.
The complex $F_{k,\ell,-}$ is obtained from the above union by deleting the $(\ell+1)$-simplex $0*[\ell+1]$ and adding the $\ell$-simplex $0*[\ell]$.
Since $\ell>1$, by the inductive hypothesis there is a PL almost embedding $f:F_{k,\ell-1,-}\to\R^{k+\ell}$ such that $|\lk f|=2z+1$.
Extend it to a map $f':F_{k,\ell,-}\to\R^{k+\ell+1}$ as follows.
Extend $f$ conically over the cone, with the vertex in the upper half-space of $\R^{k+\ell+1}$ w.r.t. $\R^{k+\ell}$.
Map the $(\ell+1)$-faces of $0*{k+\ell+2\choose \ell+1}$ (except $0*[\ell+1]$) and the $\ell$-face to the lower half-space of $\R^{k+\ell+1}$ w.r.t. $\R^{k+\ell}$.
Any two such $(\ell+1)$- or $\ell$-faces intersect at 0.
Thus the extension $f'$ is a PL almost embedding.
Since $f'\Sigma^\ell$ is the `suspension' over $f\Sigma^{\ell-1}$ and $f'=f$ on $\Sigma^k$, we have
$|\lk f'|=|\lk f|=2z+1$.
\end{proof}

\section{
Alternative proof of Theorem \ref{t:ss1ae}}\label{s:alt}

Let $\Delta^k\subset\Sigma^k$ be the $k$-simplex with the vertex set $\{\ell+3,\ell+4,\ldots,k+\ell+3\}$.
(So that $\Delta^k\ne \Delta^{\ell+1}$ even when $k=\ell+1$.)

\begin{lemma}\label{l:ss} For any integers $0\le \ell\le k$ there is

(a) a PL map $g:F\to\R^{k+\ell+1}$ whose self-intersection set consists of two points, one in $\Int\Delta^{\ell+1}$ and the other in $\Int\Delta^k$, so that the images of these interiors intersect transversally.  \cite[Lemma~1.1]{SS92}\footnote{The condition $\ell<k$ is present in \cite[Lemma~1.1]{SS92} but is not used in the proof.}

(b) a PL embedding $f:F_-\to\R^{k+\ell+1}$  such that $|\lk f|=1$.
\end{lemma}

Part (b) follows from (a).

The {\it simplicial deleted product} of a complex $K$ is
$$K^{\times2}_\Delta:=\cup\{\sigma\times\tau\ :\ \sigma,\tau\in K,\sigma\cap\tau=\emptyset\}.$$
For a complex $K$, a map $g:K\to\R^d$ and an equivariant subset $G\subset K^2$ such that $g(x)\ne g(y)$ for each $(x,y)\in G$ define an equivariant map
$$g^{\times2}_\Delta:G\to S^{d-1}\quad\text{by}\quad g^{\times2}_\Delta(x,y) := \frac{g(x)-g(y)}{|g(x)-g(y)|}.$$
If $g$ is an almost embedding, we assume that $G=K^{\times2}_\Delta$.

\begin{proposition}\label{p:delp} Let $d$ be an integer and $K$ a $k$-complex such that either

(a) $2d\ge 3k+3$; or

(b) $d\ge k+2$ and $2d-k-3\ge\dim\alpha+\dim\beta$ for any disjoint simplices $\alpha,\beta\subset K$.

For any equivariant map $\Phi:K^{\times2}_\Delta\to S^{d-1}$ there is a PL almost embedding $f:K\to\R^d$ such that $f^{\times2}_\Delta$ is equivariantly homotopic to $\Phi$.
\end{proposition}


Part (a) is a celebrated result of Weber \cite{We67}, see survey \cite[\S5]{Sk06}.
Part (b) works for $2d<3k+3$ and is an easy corollary of the generalization \cite[Disjunction Theorem 3.1]{Sk02}.
Part (b) in some sense generalizes Theorem \ref{t:ss1ae}, see \aronly{Remark \ref{lemma:construction}.a.} \jonly{\cite[Remark 3.6.a]{KS20}.}

\begin{proof}[Proof of Proposition \ref{p:delp}.b]
Apply \cite[Disjunction Theorem 3.1]{Sk02} to $N=|K|$, $T=K$, $A=\emptyset$, $E_1=K^{\times2}_\Delta$,
$E_0=\emptyset$, $h_0$ any PL map and the given map $\Phi$.
Let $f$ be the obtained map $h_1$.
Then by \cite[(3.1.1)]{Sk02} $f$ is an almost embedding.
By \cite[(3.1.2)]{Sk02} $f^{\times2}_\Delta$ is equivariantly homotopic to $\Phi$.
\end{proof}


\begin{lemma}\label{l:delpro} For any integers $0<\ell<k$ and $z$ there is an equivariant map
$\Phi:(F_-)^{\times2}_\Delta\to S^{k+\ell}$ such that $\deg\Phi|_{\Sigma^k\times\Sigma^\ell}=2z+1$.
\end{lemma}

\begin{proof}[Proof of Theorem \ref{t:ss1ae} for $1<\ell<k$ modulo Lemma \ref{l:delpro}]
Theorem \ref{t:ss1ae} for $1<\ell<k$ follows from Lemma \ref{l:delpro} and Proposition \ref{p:delp}.b because
$f^{\times2}_\Delta$ is homotopic to $\Phi$ on $\Sigma^k\times\Sigma^\ell$, so $|\lk f|=2z+1$ \cite{Sk16h}.
\end{proof}

The {\it simplicial deleted join} of a complex $K$ is
$$K^{*2}_\Delta:=\cup\{\sigma*\tau\ :\ \sigma,\tau\in K,\sigma\cap\tau=\emptyset\}.$$

\begin{lemma}\label{l:deljoi} For any integers $0\le\ell<k$  we have $F^{*2}_\Delta\cong_{\Z_2} S^{k+\ell+2}$.
\end{lemma}

\begin{proof} \emph{A subset $\sigma\subset\{0,1,\ldots,k+\ell+3\}$ is a face of $F$ if and only if the complement $\overline\sigma$ is not a face of $F$.}
Indeed,

$\bullet$ if $0\in\sigma$, then both claims are equivalent to $|\sigma|\le \ell+2$;

$\bullet$ if $0\not\in\sigma$, then both claims are equivalent to $|\sigma|\le k+1$.

This property ($F$ is {\it Alexander dual to itself}) implies that $F^{*2}_\Delta\cong_{\Z_2} S^{k+\ell+2}$
by a result of Bier \cite[Definition 5.6.1 and Theorem 5.6.2]{Ma03}.
\end{proof}

\begin{lemma}\label{l:pseudomanifold}
(a) Any two $(k+\ell+1)$-cells of $F^{\times2}_\Delta$ can be joined by a sequence of $(k+\ell+1)$-cells of $F^{\times2}_\Delta$ in which any two consecutive $(k+\ell+1)$-cells have a common $(k+\ell)$-cell.

(b) Any $(k+\ell)$-cell of $F^{\times2}_\Delta$ belongs to precisely two $(k+\ell+1)$-cells of $F^{\times2}_\Delta$.

(c) There is a collection of orientations on $(k+\ell+1)$-cells of $F^{\times2}_\Delta$ such that for any $(k+\ell)$-cell of $F^{\times2}_\Delta$ the orientations on the two adjacent $(k+\ell+1)$-cells of $F^{\times2}_\Delta$ induce the opposite orientations on the $(k+\ell)$-cell.

(d) For the orientations on $(k+\ell+1)$-cells of $F^{\times2}_\Delta$ given by (b) the exchange $\pi(x,y):=(y,x)$ of the factors acts on the orientations as multiplication by $(-1)^{k+\ell}$.
\end{lemma}

\begin{proof} This follows by Lemma \ref{l:deljoi} because the formula $\sigma\times\tau\mapsto\sigma*\tau$
defines a 1--1 correspondence between $(p+1)$-cells of $F^{*2}_\Delta$ and $p$-cells of $F^{\times2}_\Delta$, $p>0$, which respects adjacency and orientation.
For part (d) we also need that the antipodal involution of $S^{k+\ell+2}$ multiplies the orientation by $(-1)^{k+\ell+1}$.
\end{proof}

\begin{proof}[Proof of Lemma \ref{l:delpro}]
Denote by $F^{\times2,(k+\ell)}_\Delta$ the $(k+\ell)$-skeleton  of $F^{\times2}_\Delta$.
Take a map $g$ given by Lemma \ref{l:ss}.a.
Then $g^{\times2}_\Delta:F^{\times2,(k+\ell)}_\Delta\to S^{k+\ell}$ is defined.
We have $\deg g^{\times2}_\Delta|_{\Sigma^k\times\Sigma^\ell}=\lk g|_{F_-}=\pm1$.

Informally, the lemma now follows because $(F_-)^{\times2}_\Delta$ is obtained from the connected pseudomanifold $F^{\times2}_\Delta$ by deleting two codimension 0 submanifolds $\Delta^{\ell+1}\times\Sigma^k$ and $\Sigma^k\times\Delta^{\ell+1}$, which go one to the other under the exchange of factors.

Formally, we shall modify the map $g^{\times2}_\Delta$ as follows.

For an integer $a$, an equivariant map $\Psi:F^{\times2,(k+\ell)}_\Delta\to S^{k+\ell}$ and oriented $(k+\ell)$-cell $V$ of $F$ denote by $\Psi_{V,a}:F^{\times2,(k+\ell)}_\Delta\to S^{k+\ell}$ any equivariant map obtained by the following construction (in fact, this construction produces a map $\Psi_{V,a}$ well-defined up to homotopy).
Define $\Psi_{V,a}|_V$ to be the connected sum of $\Psi|_V$ and a map $S^{k+\ell}\to S^{k+\ell}$ of degree $a$.
(In other words, define $\Psi_{V,a}|_V$ to be the composition
$V\overset c\to V\vee S^{k+\ell}\overset{\Psi\vee\widehat a}\to S^{k+\ell}$, where $c$ is the contraction of certain $(k+\ell-1)$-sphere in the interior of $V$ and $\widehat a$ is a map of degree $a$.)
Define $\Psi_{V,a}(x,y):=-\Psi_{V,a}|_V(y,x)$ for $(y,x)\in V$.
Define $\Psi_{V,a}=\Psi$ elsewhere.
We write that $\Psi_{V,a}$ is obtained from $\Psi$ by {\it the modification  $(V,a)$.}

For oriented manifolds $A$ and $B$ of the same dimension denote $[A:B]=+1$ if $B\subset A$ and their orientations coincide, $[A:B]=-1$ if $B\subset A$ and their orientations are the opposite, and $[A:B]=0$ otherwise (i.e., if $B\not\subset A$).

Clearly, $\deg\Psi_{V,a}|_{\partial A}=\deg\Psi|_{\partial A}$ for any $(k+\ell+1)$-cell $A$ disjoint from $V\cup\pi V$.
For a $(k+\ell+1)$-cell $U\supset V$ we have
$$\deg\Psi_{V,a}|_{\partial U} - \deg\Psi|_{\partial U} = [\partial U:V] a.$$
By Lemma \ref{l:pseudomanifold}.d and since the antipodal involution of $S^{k+\ell}$ multiplies the orientation by $(-1)^{k+\ell+1}$, we have
$$\deg\Psi_{V,a}|_{\pi\partial U} - \deg\Psi|_{\pi\partial U} =- [\partial U:V] a.$$
By Lemma \ref{l:pseudomanifold}.a there is a sequence
$\Delta^k\times\Delta^{\ell+1} = U_0,U_1,\ldots,U_m = \Delta^{\ell+1}\times\Delta^k$ of $(k+\ell+1)$-cells of $F^{\times2}_\Delta$ such that $V_i:=U_{i-1}\cap U_i$ is a $(k+\ell)$-cell for each $i=1,2,\ldots,m$.
Take the above orientations on the $U_i$ and orient the $V_i$ so that $[\partial U_i:V_i]=1$ for each $i=1,2,\ldots,m$.
Denote by $\Phi:F^{\times2,(k+\ell)}_\Delta\to S^{k+\ell}$ any equivariant map obtained from $g^{\times2}_\Delta$
by the modifications $(V_1,-z),\ldots,(V_m,-z)$.
Clearly, $\deg\Phi|_{\partial A}=\deg g^{\times2}_\Delta|_{\partial A}$ for any $(k+\ell+1)$-cell $A\not\in\{U_0,U_m\}$.
Then $\deg\Phi|_{\partial U_0} - \deg g^{\times2}_\Delta|_{\partial U_0}  = 2z$.
We have $\deg g^{\times2}_\Delta|_{\Sigma^k\times\Sigma^\ell}=\pm\lk g|_{F_-}$.
If this degree is $+1$, then we are done.
If this degree is $-1$, then we make additionally the same construction replacing $-z$ by $-1$.
\end{proof}

\aronly{

\begin{remark}\label{lemma:construction}
(a) Let us give a direct proof of Theorem \ref{t:ss1ae} for $1<\ell<k$ without reference to Proposition \ref{p:delp}.b.
Apply \cite[Disjunction Theorem 3.1]{Sk02}\footnote{The required particular case of this result can easily be recovered by `turning some simplices around another simplices' and Whitney trick, see the survey \cite[beginning of \S8]{Sk06}.} to $N=|F_-|$, $T=F_-$, $A=\emptyset$, $E_1=(F_-)^{\times2}_\Delta$,
$E_0=\emptyset$, $h_0=g$ and the map $\Phi$ given by Lemma \ref{l:delpro}.
Observe that for disjoint simplices $\alpha,\beta\subset F_-$ we have $\dim\alpha+\dim\beta\le k+\ell+1$.
Since $l\ge2$, we have
$$k+\ell+1+\dim F_-=2k+\ell+1\le2(k+\ell+1)-3\quad\text{and}\quad\dim F_-=k\le(k+\ell+1)-2.$$
Thus the assumptions of \cite[Disjunction Theorem 3.1]{Sk02} are fulfilled.
Let $f$ be the obtained map $h_1$.
Then by \cite[(3.1.1)]{Sk02} $f$ is an almost embedding.
By \cite[(3.1.2)]{Sk02} $f^{\times2}_\Delta$ is homotopic to $\Phi$ on $\Sigma^k\times\Sigma^\ell$.
So $|\lk f|=2z+1$ \cite{Sk16h}.

(b) Analogously to Lemma \ref{l:delpro} one proves the following.
Assume that $c$ is an assignment of integers to $(\ell+k+1)$-cells $U\subset F^{\times2}_\Delta$ oriented as above
(i.e., for any $c\in Z^{\ell+k+1}(F^{\times2}_\Delta;\Z)$) such that
$\sum_U |c(U)|\equiv 2 \mod 4$ and $c(\pi(U)) = -c(U)$ for any $U$.
Then there is an equivariant map $\Phi:F^{\times2,(k+\ell)}_\Delta\to S^{k+\ell}$ such that
$\deg\Phi|_{\partial U}=c(U)$ for any $(\ell+k+1)$-cell $U\subset F^{\times2}_\Delta$.
\end{remark}

\section{Appendix: a direct proof of Lemma \ref{l:pseudomanifold}}\label{s:pseu}

\begin{proof}[Direct proof of Lemma \ref{l:pseudomanifold}.a]
Any $(k+\ell+1)$-cell $\alpha\times\beta \subset {[k+\ell+3]\choose\le k+1}^{\times2}_\Delta$ has a common $(k+\ell)$-cell with another $(k+\ell+1)$-cell $\alpha'\times\beta'$ whenever $\alpha'\supset\alpha$, $\beta'\subset\beta$ and $|\alpha'-\alpha| = |\beta-\beta'|=1$.
In several such steps we join any two $(k+\ell+1)$-cells in ${[k+\ell+3]\choose \le k+1}^{\times2}_\Delta$.

For any $(k+\ell+1)$-cell $(0,\alpha)\times\beta$ we have $\dim\alpha=\ell$ and $\dim\beta=k$.
This cell has a common $(k+\ell)$-cell with
$(v,\alpha)\times\beta \subset {[k+\ell+3]\choose\le k+1}^{\times2}_\Delta$, where $\{v\}=[k+\ell+3]-\alpha-\beta$.
The case when $0$ is in the second factor is analogous.
Hence any cell involving $0$ is also joined to any other cell.
\end{proof}

{\bf Definition of orientations on $(\ell+k+1)$-cells of $F^{\times2}_\Delta$.}
If $(a_0,\ldots,a_d)$ is a sequence of distinct elements of $\{0,1,\ldots, \ell+k+3\}$ then $(a_0,\ldots,a_d)$ is understood as a simplex of $F$ with the orientation induced by this order.
We assign the orientations to $(\ell+k+1)$-cells of $\widetilde F$ by modifying the product orientations
as follows.
In the following formulas $\sgn$ is the sign of a permutation of $[\ell+k+3]$.
For $\alpha=(a_0,\ldots,a_m)$ and $\beta=(b_0,\ldots, b_n)$ let $(\alpha,\beta):=(a_0,\ldots,a_m,b_0,\ldots,b_n)$.
The notation $(v,\alpha,\beta)$ and $(\alpha,v,\beta)$ has analogous meaning.

(1) for $\alpha\times\beta$, where $\alpha\sqcup\beta=[\ell+k+3]$, we modify by $\sgn(\alpha,\beta)$;

(2) for $(0,\alpha)\times\beta$, where $\dim\alpha=\ell$, we modify by $-\sgn(v,\alpha,\beta)$, where $\{v\}=[\ell+k+3]-\alpha-\beta$.

(2') for $\alpha\times (0,\beta)$, where $\dim\beta=\ell$, we modify by
$-\sgn(\alpha,v,\beta)$, where $\{v\}=[\ell+k+3]-\alpha-\beta$.

The orientations are well defined because permuting the elements of $\alpha$ we obtain the same change in the product orientation and in the sign of the total permutation; the same applies to $\beta$.

\begin{proof}[Direct proof of Lemma \ref{l:pseudomanifold}.bc]
We need to check that any $(\ell+k)$-cell $\sigma\times \tau$ of $F^{\times2}_\Delta$ belongs to precisely two $(\ell+k+1)$-cells $A,B$ of $F^{\times2}_\Delta$, and that
$[\partial A:\sigma\times\tau]+[\partial B:\sigma\times\tau]=0$ for the above orientations on $A,B$.
Denote $[X]:=[X:\sigma\times\tau]$ for brevity.
In the following formulas the products are oriented as the products, so the above orientation is indicated by the sign.
Recall that for $(\ell+k+1)$-cell $\alpha\times\beta$  oriented as the product we have
$$
[\partial(\alpha\times\beta)] =
[\partial\alpha:\sigma][\beta:\tau]+ (-1)^{\dim\alpha}[\alpha:\sigma][\partial\beta:\tau].
$$
\emph{Case $\sigma\times\tau \subset {[\ell+k+3]\choose \le k+1}^{\times2}_\Delta$, where $\ell<\dim\sigma,\dim\tau<k$.}
Let $v:=[k+\ell+3]-\sigma-\tau$.
The cell $\sigma\times\tau$ is in the boundary of only
$$
A:=\sgn(v,\sigma,\tau)(v,\sigma)\times\tau\quad\text{and}\quad B:=\sgn(\sigma,v,\tau)\sigma\times(v,\tau),\quad\text{so}
$$
$$
[\partial A]+[\partial B] = \sgn(v,\sigma,\tau) + \sgn(\sigma,v,\tau) (-1)^{\dim\sigma} = 0.
$$
\emph{Case $\sigma\times\tau \subset {[\ell+k+3]\choose \le k+1}^{\times2}_\Delta$, where $\dim\sigma=\ell$ and $\dim\tau=k$.}
Let $v:=[k+\ell+3]-\sigma-\tau$.
The cell $\sigma\times\tau$ is in the boundary of only
$$
A:=\sgn(v,\sigma,\tau)(v,\sigma)\times\tau\quad\text{and}\quad B:=\sgn(v,\sigma,\tau)(0,\sigma)\times\tau,\quad\text{so}
$$
$$
[\partial A]+[\partial B] =  \sgn(v,\sigma,\tau) - \sgn(v,\sigma,\tau) = 0.
$$
\emph{Case $\sigma\times\tau \subset {[\ell+k+3]\choose\le k+1}^{\times2}_\Delta$, where $\dim\sigma=k$ and $\dim\tau=\ell$.}
Let $v:=[k+\ell+3]-\sigma-\tau$.
The cell $\sigma\times\tau$ is in the boundary of only
$$
A:=\sgn(\sigma,v,\tau)\sigma\times(v,\tau)\quad\text{and}\quad B:=-\sgn(\sigma,v,\tau)\sigma\times(0,\tau),\quad\text{so}
$$
$$
[\partial A]+[\partial B]=\sgn(\sigma,v,\tau) (-1)^k - \sgn(\sigma,v,\tau) (-1)^k = 0.
$$
\emph{Case $\sigma\times\tau = (0,\sigma')\times\tau$, where $\dim\sigma'=\ell-1$ and $\dim\tau=k$.}
Let $\{v, w\}:=[k+\ell+3]-\sigma'-\tau$, $v<w$.
The cell $\sigma\times\tau$ is in the boundary of only
$$
A:=-\sgn(w,v,\sigma',\tau)(0, v,\sigma')\times\tau \quad\text{and}\quad
B:=-\sgn(v, w, \sigma',\tau)(0, w, \sigma')\times\tau,\quad\text{so}
$$
$$
[\partial A]+[\partial B] = \sgn(w,v,\sigma',\tau) + \sgn(v,w,\sigma',\tau) = 0.
$$
\emph{Case $\sigma\times\tau = \sigma\times(0,\tau')$, where $\dim\sigma=k$ and $\dim\tau'=\ell-1$.}
Let $\{v, w\}:=[k+\ell+3]-\sigma-\tau'$, $v<w$.
The cell $\sigma\times\tau$ is in the boundary of only
$$
A:=- \sgn(\sigma, w, v, \tau')\sigma\times(0, v, \tau')\quad\text{and}\quad
B:=- \sgn(\sigma, v, w,\tau')\sigma\times(0, w, \tau'),\quad\text{so}
$$
$$
[\partial A]+[\partial B]= - \sgn(\sigma, w, v, \tau') (-1)^{k+1} - \sgn(\sigma, v, w, \tau') (-1)^{k+1} = 0.
$$
\end{proof}

\begin{proof}[Direct proof of Lemma \ref{l:pseudomanifold}.d]
By Lemma \ref{l:pseudomanifold}.a it suffices to consider only one $(k+\ell+1)$-cell of our choice.
Choose a cell $\sigma\times\tau \subset {[\ell+k+3]\choose k+1}^{\times2}_\Delta$.
It is oriented as a product with the sign $\sgn(\sigma,\tau)$.
The cell $\tau\times\sigma$ is oriented as the product with the sign
$$
\sgn(\tau,\sigma) = (-1)^{(k+1)(\ell+2)}\sgn(\sigma,\tau) = (-1)^{k\ell + \ell}\sgn(\sigma,\tau).
$$
The exchange of the factors acts on the product orientations as multiplication by $(-1)^{\ell k + k}$.
Hence (d) follows.
\end{proof}

\section{Appendix: an explicit proof of Lemma \ref{l:delpro}}\label{s:delpro}

Here we present an explicit construction for the proof of Lemma \ref{l:delpro} for $k$ odd and $\ell$ even.
This is nothing but giving explicit $U_0,U_1,\ldots,U_m$ from the proof in \S\ref{s:pr}.
However, not constructing the orientations makes this proof shorter than the proof via \S\ref{s:pseu}.

\begin{proof}[Proof: construction of $\Phi'$ for $k$ odd and $\ell$ even]
Since $\ell$ is even, it suffices to prove the lemma for $\Sigma^k\times\Sigma^\ell$ replaced by $\Sigma^\ell\times\Sigma^k$.
Take an embedding $f$ given by Lemma \ref{l:ss}.b.
Take the equivariant map $f^{\times2}_\Delta:(F_-)^{\times2}_\Delta\to S^{k+\ell}$.
We have $|\deg f^{\times2}_\Delta|_{\Sigma^\ell\times\Sigma^k}| = |\lk f| = 1$.
We may assume that the degree is $+1$, otherwise we make the construction below replacing $z$ by $z+1$.
We shall modify the map $f^{\times2}_\Delta$ on the $(k+\ell)$-skeleton of $(F_-)^{\times2}_\Delta$ as follows.

For an integer $a$, an equivariant map $\Psi:(F_-)^{\times2,(k+\ell)}_\Delta\to S^{k+\ell}$ and oriented simplices $\sigma,\tau$ of $F_-$ the sum of whose dimensions is $k+\ell$ denote by $\Psi_{\sigma,\tau,a}:(F_-)^{\times2,(k+\ell)}_\Delta\to S^{k+\ell}$ any equivariant map obtained by the following construction (in fact, this construction produces a map $\Psi_{\sigma,\tau,a}$ well-defined up to homotopy).
Define $\Psi_{\sigma,\tau,a}|_{\sigma\times\tau}$ to be the connected sum of  $\Psi|_{\sigma\times\tau}$ and a map $S^{k+\ell}\to S^{k+\ell}$ of degree $a$.
(In other words, define $\Psi_{\sigma,\tau,a}|_{\sigma\times\tau}$ to be the composition
$\sigma\times\tau\overset c\to\sigma\times\tau\vee S^{k+\ell}\overset{\Psi\vee a}\to S^{k+\ell}$, where $c$ is the contraction of certain $(k+\ell-1)$-sphere in the interior of $\sigma\times\tau$ and $a$ is a map of degree $a$.)
Define $\Psi_{\sigma,\tau,a}|_{\tau\times\sigma}(x,y):=-\Psi_{\sigma,\tau,a}|_{\sigma\times\tau}(y,x)$.
Define $\Psi_{\sigma,\tau,a}=\Psi$ elsewhere.
We write that $\Psi_{\sigma,\tau,a}$ is obtained from $\Psi$ by {\it the modification  $(\sigma,\tau,a)$.}

For a sequence $(a_0,\ldots,a_s)$ of distinct elements of $\{0,1,\ldots,k+\ell+3\}$ denote by $(a_0,\ldots,a_s)$
the oriented $s$-simplex of $F$ with vertices $a_0,\ldots,a_s$ and the orientation induced by this order.
For $m=\ell, \ldots,k$ and $j=0,\ldots,\ell $ let
$$\sigma_m = (1,\ldots,m+1),\quad\tau_m = (m+3,\ldots,k+\ell+3),$$
$$\phi_{2j+1} = (0,\ldots,j,k+j+3,\ldots,k+\ell+3),\quad \psi_{2j+1} = (-1)^{j(j+1)}(j+2,j+3,\ldots,k+1+j),$$
$$\phi_{2j+2} = (0,\ldots,j,k+j+4,\ldots,k+\ell+3),\quad \psi_{2j+2} = (-1)^{(j+1)^2}(j+2,j+3,\ldots,k+2+j).$$
Let $\Phi'$ be a map obtained from $f^{\times2}_\Delta|_{(F_-)^{\times2,(k+\ell)}_\Delta}$ by the modifications

$\bullet$ $(\sigma_\ell ,\tau_\ell ,z)$,

$\bullet$ $(\sigma_m,\tau_m,(-1)^mz)$ for $m=\ell+1,\ldots,k$,

$\bullet$ $(\phi_{2j+1},\psi_{2j+1},(-1)^{j+1} z)$ for $j=0,\ldots,\ell $ and

$\bullet$ $(\phi_{2j+2},\psi_{2j+2},(-1)^j z)$ for $j=0,\ldots,\ell $.

It suffices to prove  that $\Phi'$ equivariantly extends to $(F_-)^{\times2}_\Delta$ and $\deg\Phi'|_{\Sigma^\ell\times\Sigma^k}=2z+1$.
\end{proof}

\begin{proof}[Proof that $\deg\Phi'|_{\Sigma^\ell\times\Sigma^k}=2z+1$]
Clearly,
$$\deg\Psi_{\sigma,\tau,a}|_{\Sigma^\ell\times\Sigma^k}-\deg\Psi|_{\Sigma^\ell\times\Sigma^k} = a[\Sigma^\ell\times\Sigma^k:\sigma\times\tau] = a[\Sigma^\ell:\sigma][\Sigma^k:\tau].$$
Of the above bullet points modifications this is non-zero only for $(\sigma_\ell,\tau_\ell,z)$ and $(\phi_{2\ell+2},\psi_{2\ell+2},(-1)^\ell z)$, when this is $z$ and $z$ respectively.
Hence
$\deg\Phi'|_{\Sigma^\ell\times\Sigma^k}=z+z+\deg f^{\times2}_\Delta|_{\Sigma^\ell\times\Sigma^k}=2z+1$.
\end{proof}

\begin{proof}[Proof that $\Phi'$ equivariantly extends to $(F_-)^{\times2}_\Delta$]
In the rest of this proof $\alpha,\beta$ are disjoint simplices of $F_-$ the sum of whose dimensions is $k+\ell+1$.
Then either

$\bullet$ $\alpha,\beta\subset[k+\ell+3]$ and $\alpha\sqcup\beta=[k+\ell+3]$, or

$\bullet$ $\dim\alpha=\ell+1$, $0\in\alpha$, $\beta\in {[k+\ell+3]\choose k+1}$, and $(k+\ell+3)-|\alpha|-|\beta|=1$.

Hence in the above bullet points modifications

$\bullet$ $\sigma_\ell \times\tau_\ell \subset\alpha\times\beta$ only if  $\alpha\times\beta=\sigma_{\ell+1}\times\tau_\ell $;

$\bullet$ for $m=\ell+1,\ldots,k-1$ we have
$\sigma_m\times\tau_m\subset\alpha\times\beta$ only if $\alpha\times\beta\in\{\sigma_m\times\tau_{m-1},\sigma_{m+1}\times\tau_m\}$;

$\bullet$ $\sigma_k\times\tau_k\subset\alpha\times\beta$ only if
$\alpha\times\beta\in\{\sigma_k\times\tau_{k-1},\phi_1\times\psi_0\}$;

$\bullet$ for $j=0,1,\ldots,\ell $ we have $\phi_{2j+1}\times\psi_{2j+1}\subset\alpha\times\beta$ only if $\alpha=\varphi_{2j+1}$, $\beta\in\{\psi_{2j},\psi_{2j+2}\}$;

$\bullet$ for $j=0,1,\ldots,\ell -1$ we have $\phi_{2j+2}\times\psi_{2j+2}\subset\alpha\times\beta$ only if   $\alpha\in\{\varphi_{2j+1},\varphi_{2j+3}\}$, $\beta=\psi_{2j+2}$;

$\bullet$ $\phi_{2\ell+2}\times\psi_{2\ell+2}\subset\alpha\times\beta$ only if $\alpha\times\beta=\varphi_{2\ell+1}\times\psi_{2\ell+2}$.


Now assume that $\alpha,\beta$ are oriented.
We consider the product orientations on cells of $(F_-)^{\times2}_\Delta$.
It suffices to prove that for each cell $\alpha\times\beta$ appearing in the above list  $\deg\Phi'|_{\partial(\alpha\times\beta)}=0$.
Observe that $\deg f^{\times2}_\Delta|_{\partial(\alpha\times\beta)}=0$.
We define boundary so that $[\partial(a_0,a_1,\ldots,a_m):(a_1,\ldots,a_m)]=+1$.
Then for $\alpha\supset\sigma,\beta\supset\tau$ we have
$$\deg\Psi_{\sigma,\tau,a}|_{\partial(\alpha\times\beta)}-\deg\Psi|_{\partial(\alpha\times\beta)} = a[\partial(\alpha\times\beta):\sigma\times\tau] =
a\left([\partial\alpha:\sigma][\beta:\tau]+ (-1)^{\dim\alpha}[\alpha:\sigma][\partial\beta:\tau]\right).$$
The restriction to $\partial(\sigma_m\times \tau_{m-1})$, for $m=\ell+1,\ldots,k$, receives two modifications in  the construction of $\Phi'$: on $\sigma_{m-1}\times\tau_{m-1}$ and on $\sigma_m\times\tau_m$.
Hence
$$\deg\Phi'|_{\partial(\sigma_m\times \tau_{m-1})} -
\deg f^{\times2}_\Delta|_{\partial(\sigma_m\times \tau_{m-1})}=$$
$$=(-1)^{m-1}z [\partial\sigma_m:\sigma_{m-1}][\tau_{m-1}:\tau_{m-1}]+
(-1)^{m}z (-1)^{m}[\sigma_m:\sigma_m][\partial\tau_{m-1}:\tau_m] =$$
$$=(-1)^{m-1+m}z + (-1)^{m+m}z =0.$$
In order to work with the signs in the sequel, it is convenient to rewrite the definitions of $\psi_*$  with different order of vertices and no signs in front:
$$
\psi_{2j+1} = (k+2,\ldots,k+1+j,j+2,\ldots,k+1),\quad \psi_{2j+2} = (k+2,\ldots,k+2+j,j+2,\ldots,k+1).
$$
Observe that $\psi_0 = \sigma_k$.
Hence the restriction to $\partial(\phi_1\times \psi_0)$ receives two modifications in the construction of $\Phi'$: on $\sigma_k\times\tau_k$ and on $\phi_1\times\psi_1$.
Hence
$$\deg\Phi'|_{\partial(\phi_1\times \psi_0)} - \deg f^{\times2}_\Delta|_{\partial(\phi_1\times \psi_0)} =$$
$$=-z(-1)^{\ell+1}(-1)^{\ell+1}[\partial\phi_1:\tau_k][\psi_0:\sigma_k]- z(-1)^{\ell+1}[\phi_1:\phi_1][\partial\psi_0:\psi_1]
=-z+z=0.$$
The restriction to $\partial(\phi_{2j+1}\times\psi_{2j+2})$, for $j=0, \ldots, l$, receives two modifications in the construction of $\Phi'$: on $\phi_{2j+1}\times\psi_{2j+1}$ and on $\phi_{2j+2}\times\psi_{2j+2}$.
Hence
$$\deg\Phi'|_{\partial(\phi_{2j+1}\times\psi_{2j+2})} -
\deg f^{\times2}_\Delta|_{\partial(\phi_{2j+1}\times\psi_{2j+2})}=$$
$$=(-1)^{j+1} z (-1)^{\ell+1}[\phi_{2j+1}:\phi_{2j+1}][\partial\psi_{2j+2}:\psi_{2j+1}] +
(-1)^j z [\partial\phi_{2j+1}:\phi_{2j+2}][\psi_{2j+2}:\psi_{2j+2}]=$$
$$= (-1)^{j+1+\ell+1+j}z + (-1)^{j+1+j}z=0.$$
The restriction to $\partial(\phi_{2j+1}\times \psi_{2j})$, for $j=1, \ldots, l$, receives two modifications in the construction of $\Phi'$: on $\phi_{2j}\times\psi_{2j}$ and on $\phi_{2j+1}\times\psi_{2j+1}$.
Hence
$$\deg\Phi'|_{\partial(\phi_{2j+1}\times\psi_{2j})} -
\deg f^{\times2}_\Delta|_{\partial(\phi_{2j+1}\times\psi_{2j})}=$$
$$=(-1)^{j+1} z [\partial\phi_{2j+1}:\phi_{2j}][\psi_{2j}:\psi_{2j}]+(-1)^{j+1} z (-1)^{\ell+1}[\phi_{2j+1}:\phi_{2j+1}][\partial\psi_{2j}:\psi_{2j+1}] =$$
$$=(-1)^{j+1+j}z +(-1)^{j+1+\ell+1+j}z=0.$$
Thus indeed for each cell $\alpha\times\beta$ appearing in the above list  $\deg\Phi'|_{\partial(\alpha\times\beta)}=0$.
\end{proof}

}

{\it In this list books, surveys and expository papers are marked by stars}


\begin{thebibliography}{99}

\UseRawInputEncoding

\newcommand{\abc}{\bibitem[ABC+]{ABC+} * \emph{M. Atiyah, A. Borel, G. J. Chaitin, D. Friedan, J. Glimm, J. J. Gray, M. W. Hirsch, S. MacLane, B. B. Mandelbrot, D. Ruelle, A. Schwarz, K. Uhlenbeck, R. Thom, E. Witten, C.  Zeeman.} Responses to ``Theoretical Mathematics: Toward a cultural synthesis of mathematics and theoretical physics'', by A. Jaffe and F. Quinn. Bull. Am. Math. Soc. 30 (1994) 178--207. arXiv:math/9404229.}

\newcommand{\agles}{\bibitem[AGL]{AGL86} Mathematical Economics,  ed. by A. Ambrosetti, F. Gori, R. Lucchetti,
Lect. Notes Math. 1330, Springer, 1986.}


\newcommand{\akzz}{\bibitem[Ak00]{Ak00} * \emph{П. М. Ахметьев.} Вложения компактов, стабильные
гомотопические группы сфер и теория особенностей, Успехи Мат. Наук.  2000. 55:3. C.~3-62.}

\newcommand{\akoe}{\bibitem[AK19]{AK19} \emph{S. Avvakumov, R. Karasev.} Envy-free division using mapping degree. arXiv:1907.11183.}

\newcommand{\akto}{\bibitem[AK21]{AK21} \emph{G. Arone and V. Krushkal.}
Embedding obstructions in $\R^d$ from the Goodwillie-Weiss calculus and Whitney disks. arXiv:2101.10995. }

\newcommand{\akm}{\bibitem[AKM]{AKM} \emph{M. Abrahamsen, L. Kleist and T. Miltzow.}
Geometric Embeddability of Complexes is $\exists\mathbb R$-complete, arXiv:2108.02585.}

\newcommand{\aksoe}{\bibitem[AKS]{AKS} \emph{S. Avvakumov, R. Karasev and A. Skopenkov.} Stronger counterexamples to the topological Tverberg conjecture, submitted. arXiv:1908.08731.}

\newcommand{\akuoe}{\bibitem[AKu19]{AKu19} \emph{S. Avvakumov, S. Kudrya.}
Vanishing of all equivariant obstructions and the mapping degree.
Discr. Comp. Geom., 66:3 (2021) 1202--1216. arXiv:1910.12628.}

\newcommand{\alto}{\bibitem[Al22]{Al22} \emph{E. Alkin,}
Hardness of almost embedding simplicial complexes in $\R^d$, II.}

\newcommand{\amsw}{\bibitem[AMS+]{AMSW} \emph{S. Avvakumov, I. Mabillard, A. Skopenkov and U. Wagner.}
Eliminating Higher-Multiplicity Intersections, III. Codimension 2, Israel J. Math. 245 (2021) 501--534.  arxiv:1511.03501.}


\newcommand{\anzt}{\bibitem[An03]{An03} * \emph{Д. В. Аносов.} Отображения окружности, векторные поля и их применения. М: МЦНМО, 2003.}

\newcommand{\arnf}{\bibitem[Ar95]{Ar95} \emph{V. I. Arnold,}  Topological invariants of plane curves and caustics, University Lecture Series, Vol. 5, Amer. Math. Soc., Providence, RI, 1995.}

\newcommand{\arszo}{\bibitem[ARS01]{ARS01} \emph{P. Akhmetiev, D. Repov\v s and A. Skopenkov},
Embedding products of low-dimensional manifolds in $\R^m$, Topol. Appl. 113 (2001), 7--12.}

\newcommand{\arszt}{\bibitem[ARS02]{ARS02} \emph{P. Akhmetiev, D. Repovs and A. Skopenkov.} Obstructions to approximating maps of $n$-manifolds into $R^{2n}$ by embeddings, Topol. Appl., 123 (2002), 3--14.}

\newcommand{\asoed}{\bibitem[As]{As} \emph{A. Asanau,} \lowercase{A SIMPLE PROOF THAT CONNECTED SUM OF ORDERED
ORIENTED LINKS IS NOT WELL-DEFINED,} Math. Notes, to appear.}

\newcommand{\asoe}{\bibitem[As]{As} \emph{A. Asanau,} On the \lowercase{TRIPLE SELF-INTERSECTION NUMBER FOR GRAPHS IN THE PLANE,} unpublished, 2018.}

\newcommand{\avos}{\bibitem[Av14]{Av14} \emph{S. Avvakumov,} The classification of certain linked 3-manifolds in 6-space, Moscow Math. J., 16:1 (2016), 1--25. arXiv:1408.3918.}


\newcommand{\bant}{\bibitem[Ba93]{Ba93} * \emph{T. Bartsch.} Topological methods for variational problems with
symmetries, Lecture Notes in Mathematics, 1560, Springer-Verlag, Berlin, 1993.}

\newcommand{\bbsn}{\bibitem[BB79]{BB} \emph{E.~G. Bajm{{\'o}}czy and I.~B{{\'a}}r{{\'a}}ny,}
\newblock On a common generalization of {B}orsuk's and {R}adon's theorem,
\newblock Acta Math.\ Acad.\ Sci.\ Hungar.\ 34:3 (1979), 347-350.}

\newcommand{\bbzos}{\bibitem[BBZ]{BBZ} * \emph{I.~B{{\'a}}r{{\'a}}ny, P.~V.~M. Blagojevi{{\'c}} and G.~M. Ziegler.} Tverberg's Theorem at 50: Extensions and Counterexamples, Notices of the Amer. Math. Soc., 63:7 (2016), 732--739.}


\newcommand{\bcm}{\bibitem[BCM]{BCM} * 13th Hilbert Problem on superpositions of functions, presented by A. Belov, A. Chilikov, I. Mitrofanov, S. Shaposhnikov and A. Skopenkov,
\url{http://www.turgor.ru/lktg/2016/5/index.htm}.}

\newcommand{\beet}{\bibitem[BE82]{BE82} * \emph{V.G. Boltyansky and V.A. Efremovich.} Intuitive Combinatorial Topology. Springer.}

\newcommand{\beetr}{\bibitem[BE82]{BE82} * \emph{В. Г. Болтянский и В. А. Ефремович.} Наглядная топология. М.:  Наука, 1982.}


\newcommand{\bfzof}{\bibitem[BFZ14]{BFZ14} \emph{P. V. M. Blagojevi{\'c}, F. Frick, and G. M. Ziegler,}
Tverberg plus constraints, Bull. Lond. Math. Soc. 46:5 (2014), 953-967, arXiv:1401.0690.}


\newcommand{\bfzos}{\bibitem[BFZ]{BFZ} \emph{P. V. M. Blagojevi{\'c}, F. Frick and G. M. Ziegler,}
Barycenters of Polytope Skeleta and Counterexamples to the Topological Tverberg Conjecture, via Constraints,
J. Eur. Math. Soc., 21:7 (2019) 2107-2116. arXiv:1510.07984.}


\newcommand{\bgos}{\bibitem[BG16]{BG16} \emph{A. Bj\"orner and A. Goodarzi}, On Codimension one Embedding of Simplicial Complexes, in book: A Journey Through Discrete Mathematics, arXiv:1605.01240.}

\newcommand{\biet}{\bibitem[Bi83]{Bi83} * \emph{R. H. Bing.} The Geometric Topology of 3-Manifolds. Providence, R.~I. 1983. (Amer. Math. Soc. Colloq. Publ., 40).}

\newcommand{\bitz}{\bibitem[Bi20]{Bi20} * \emph{A. Bikeev.} Realizability of discs with ribbons on the M\"obius strip. Mat. Prosveschenie, 28 (2021), 150-158;
erratum to appear. arXiv:2010.15833.}

\newcommand{\bitzr}{\bibitem[Bi20]{Bi20} * \emph{А. Бикеев.} Реализуемость дисков с ленточками на ленте Мебиуса.
Мат. просвещение. Сер. 3. 28 (2021), 150--158.}

\newcommand{\bito}{\bibitem[Bi21]{Bi21} {\it A. I. Bikeev,}
Criteria for integer and modulo 2 embeddability of graphs to surfaces, arXiv:2012.12070v2.}


\newcommand{\bagos}{\bibitem[BG17]{BG17} \emph{S. Basu and S. Ghosh.} Equivariant maps related to the topological Tverberg conjecture, Homology, Homotopy and Applications 19:1 (2017) 155--170.}

\newcommand{\bkkmzof}{\bibitem[BKK]{BKK} \emph{M. Bestvina, M. Kapovich and B. Kleiner,}
Van Kampen's embedding obstruction for discrete groups, Invent. Math. 150 (2002) 219--235. arXiv:math/0010141.}

\newcommand{\bmzf}{\bibitem[BM04]{BM04} \emph{Boyer, J. M. and Myrvold, W. J.} On the cutting edge: simplified $O(n)$ planarity by edge addition,  Journal of Graph Algorithms and Applications, 8:3 (2004) 241--273.}

\newcommand{\bm}{\bibitem[BM15]{BM15} \emph{I. Bogdanov and A. Matushkin.} Algebraic proofs of linear versions of the Conway--Gordon--Sachs theorem and the van Kampen--Flores theorem, arXiv:1508.03185.}


\newcommand{\bmzzn}{\bibitem[BMZ09]{BMZ09} \emph{P. V. M. Blagojevi{\'c}, B. Matschke, G. M. Ziegler,}
Optimal bounds for a colorful Tverberg-Vre\'cica type problem, Advances in Math., 226 (2011), 5198-5215, arXiv:0911.2692.}

\newcommand{\bmzof}{\bibitem[BMZ15]{BMZ15} \emph{P. V. M. Blagojevi{\'c}, B. Matschke, G. M. Ziegler,}
Optimal bounds for the colored Tverberg problem, J. Eur. Math. Soc.,  17:4 (2015) 739--754,
arXiv:0910.4987.}

\newcommand{\bpns}{\bibitem[BP97]{BP97} * \emph{R. Benedetti and C. Petronio.} Branched standard spines of 3-manifolds, Lecture Notes in Math. 1653, Springer-Verlag, Berlin-Heidelberg-New York, 1997.}

\newcommand{\brst}{\bibitem[Br72]{Br72} \emph{J. L. Bryant.} Approximating embeddings of polyhedra in codimension 3, Trans. Amer. Math. Soc., 170 (1972) 85--95.}

\newcommand{\brts}{\bibitem[Br68]{Br68} \emph{P. Bruegel,} 1568,
\url{https://en.wikipedia.org/wiki/The_Magpie_on_the_Gallows}.}


\newcommand{\bren}{\bibitem[Br82]{brown1982} * \emph{K.~S. Brown.} \newblock Cohomology of Groups. \newblock Springer-Verlag New York, 1982.}


\newcommand{\bssos}{\bibitem[BS17]{BS17} * \emph{I.~B\'{a}r\'{a}ny and P. Sober\'{o}n,} Tverberg's theorem is 50 years old: a survey, Bull. Amer. Math. Soc. (N.S.) 55:4 (2018), 459--492. arXiv:1712.06119.}

\newcommand{\bsto}{\bibitem[BS21]{BS21} * \emph{A. Buchaev and A. Skopenkov,} Simple proofs of estimations of Ramsey numbers and of discrepancy, Mat. Prosveschenie, to appear, arXiv:2107.13831.}

\newcommand{\brsnn}{\bibitem[BRS99]{BRS99} \emph{D. Repov\v s, N. Brodsky and A. B. Skopenkov.}
A classification of 3-thickenings of 2-polyhedra, Topol. Appl. 1999. 94. P.~307-314.}

\newcommand{\bsseo}{\bibitem[BSS]{BSS} \emph{I.~B\'{a}r\'{a}ny, S.~B. Shlosman, and A.~Sz{\H{u}}cs,}
\newblock On a topological generalization of a theorem of {T}verberg,
\newblock J.\ London Math.\ Soc.\ (II. Ser.) 23 (1981), 158--164.}

\newcommand{\btzs}{\bibitem[BT07]{BT07} \emph{A. Bj\"orner, M. Tancer}, Combinatorial Alexander Duality --- a Short and Elementary Proof, Discr. and Comp. Geom., 42 (2009) 586. arXiv:0710.1172.}

\newcommand{\buse}{\bibitem[Bu68]{Bu68} \emph{A. R. Butz,} Space filling curves and mathematical programming, Information and Control, 12:4 (1968) 314--330.}


\newcommand{\bz}{\bibitem[BZ16]{BZ16} * \emph{P. V. M. Blagojevi\'c and G. M. Ziegler,} Beyond the Borsuk-Ulam theorem: The topological Tverberg story, in: A Journey Through Discrete Mathematics, Eds. M. Loebl,
J. Ne\v set\v ril, R. Thomas, Springer, 2017, 273--341. arXiv:1605.07321v3.}



\newcommand{\cano}{\bibitem[Ca91]{Ca91} * \emph{D. de Caen}, The ranks of tournament matrices, Amer. Math. Monthly, 98:9 (1991) 829--831.}

\newcommand{\ca}{\bibitem[Ca]{Ca} \emph{J. Carmesin.} Embedding simply connected 2-complexes in 3-space, I-V, arXiv:1709.04642, arXiv:1709.04643, arXiv:1709.04645, arXiv:1709.04652, arXiv:1709.04659.}

\newcommand{\cfsz}{\bibitem[CF60]{CF60} \emph{P. E. Conner and E. E. Floyd}, Fixed points free involutions and equivariant maps, Bull. Amer. Math. Soc., 66 (1960) 416--441.}

\newcommand{\cget}{\bibitem[CG83]{CG83} \emph{J. H. Conway and C. M. A. Gordon},
Knots and links in spatial graphs, J. Graph Theory  7 (1983), 445--453.}

\newcommand{\cten}{\bibitem[Ch]{Ch} \emph{Chuang Tzu,} translated by H. A. Giles, Bernard Quaritch, London, 1889.}

\newcommand{\ctruku}{\bibitem[Ch]{Ch} \emph{Chuang Tzu,} translated to Russian by S. Kuchera, in: Ancient Chinese Philosophy, v. I, Mysl, Moscow, 1972.}


\newcommand{\chnn}{\bibitem[Ch99]{Ch99} * \emph{А. В. Чернавский,} Теорема Жордана.  Мат. Просвещение, 3 (1999), 142--157.}

\newcommand{\hcon}{\bibitem[HC19]{HC19} * \emph{C. Herbert Clemens.} Two-Dimensional Geometries. A Problem-Solving Approach, Amer. Math. Soc., 2019.}

\newcommand{\ckmoo}{\bibitem[CKMS]{CKMS} \emph{M. \v Cadek, M. Kr\v c\'al. J. Matou\v sek, F. Sergeraert,
L. Vok\v r\'inek, U. Wagner.} Computing all maps into a sphere, J. of the ACM, 61:3 (2014). arXiv:1105.6257.}


\newcommand{\ckmvwot}{\bibitem[CKM12+]{CKM12+} \emph{M. \v Cadek, M. Kr\v c\'al. J. Matou\v sek, L. Vok\v r\'inek, U. Wagner.} Polynomial-time computation of homotopy groups and Postnikov systems in fixed dimension, SIAM J. Comput., 43:5 (2014), 1728--1780. arXiv:1211.3093.}

\newcommand{\ckmvw}{\bibitem[CKM+]{CKM+} \emph{M. \v Cadek, M. Kr\v c\'al. J. Matou\v sek, L. Vok\v r\'inek, U. Wagner.} Extendability of continuous maps is undecidable, Discr. and Comp. Geom. 51 (2014) 24--66.
arXiv:1302.2370.}

\newcommand{\ckppt}{\bibitem[CKP+]{CKP+} \emph{E. Colin de Verdi\'ere, V. Kalu\v za, P. Pat\'ak, Z. Pat\'akov\'a and M. Tancer.} A direct proof of the strong Hanani-Tutte theorem on the projective plane. Journal of Graph Algorithms and Applications, 21:5 (2017) 939--981.}

\newcommand{\cksof}{\bibitem[CKS+]{CKS+} * New ways of weaving baskets, presented by G. Chelnokov, Yu. Kudryashov, A.Skopenkov and A. Sossinsky, \url{http://www.turgor.ru/lktg/2004/lines.en/index.htm}.}

\newcommand{\ckv}{\bibitem[CKV]{CKV} \emph{M.~{\v{C}}adek, M.~Kr\v{c}\'{a}l, and L.~Vok\v{r}\'{\i}nek.}
Algorithmic solvability of the lifting-extension problem, Discr. Comp. Geom. 57 (2017), 915--965. arXiv:1307.6444.}


\newcommand{\clr}{\bibitem[CLR]{CLR} * \emph{Т. Кормен, Ч. Лейзерсон, Р. Ривест.} Алгоритмы:
построение и анализ, МЦНМО, Москва, 1999.}

\newcommand{\clreng}{\bibitem[CLR]{CLR} * \emph{T. H. Cormen, C. E.Leiserson, R. L.Rivest, C. Stein.} Introduction to Algorithms, MIT Press, 2009.}

\newcommand{\crzfru}{\bibitem[CR]{CR} * \emph{Р. Курант, Дж. Роббинс,} Что такое математика. М.: МЦНМО, 2004.}

\newcommand{\crzfen}{\bibitem[CR]{CR} * \emph{R. Courant and H. Robbins,} What is Mathematics, Oxford Univ. Press.}

\newcommand{\crsne}{\bibitem[CRS98]{CRS98} * \emph{A. Cavicchioli, D. Repov\v s and A. B. Skopenkov.}
Open problems on graphs, arising from geometric topology, Topol. Appl. 1998. 84. P.~207-226.}

\newcommand{\crsot}{\bibitem[CRS]{CRS} \emph{M. Cencelj, D. Repov\v s and M. Skopenkov,}
Classification of knotted tori in the 2-metastable dimension, Mat. Sbornik, 203:11 (2012), 1654--1681.
arxiv:math/0811.2745.}

\newcommand{\csoo}{\bibitem[CS08]{CS08} \emph{D. Crowley and A. Skopenkov.} A classification of smooth embeddings of 4-manifolds in 7-space, II, Intern. J. Math., 22:6 (2011) 731-757, arxiv:math/0808.1795.}

\newcommand{\csos}{\bibitem[CS16]{CS16} \emph{D. Crowley and A. Skopenkov,} Embeddings of non-simply-connected 4-manifolds in 7-space. I. Classification modulo knots, Moscow Math. J., 21 (2021), 43--98. arXiv:1611.04738.}


\newcommand{\csoso}{\bibitem[CS16o]{CS16o} \emph{D. Crowley and A. Skopenkov,} Embeddings of non-simply-connected 4-manifolds in 7-space. II. On the smooth classification, Proc. A of the Royal Soc. of Edinburgh 152:1 (2022), 163--181. arXiv:1612.04776.}


\newcommand{\crsk}{\bibitem[CS]{CS} \emph{D. Crowley and A. Skopenkov,} Embeddings of non-simply-connected 4-manifolds in 7-space. III. Piecewise-linear classification. draft.}

\newcommand{\cutz}{\bibitem[Cu20]{Cu20} \emph{C. Culter,} Cantor sets are not tangent homogeneous,
Topol. Appl. 271 (2020) 1--9.}


\newcommand{\dies}{\bibitem[Di87]{Di} * \emph{T. tom Dieck,} Transformation groups, Studies in Mathematics, vol. 8, Walter de Gruyter, Berlin, 1987.}

\newcommand{\dent}{\bibitem[De93]{De93}  \emph{T.K. Dey.} On counting triangulations in $d$-dimensions. Comput. Geom.  3:6 (1993) 315--325.}

\newcommand{\denf}{\bibitem[DE94]{DE94}  \emph{T.K. Dey and H. Edelsbrunner.} Counting triangle crossings and halving planes, Discrete Comput. Geom. 12 (1994), 281--289.}

\newcommand{\dgn}{\bibitem[DGN+]{DGN+} * Low rank matrix completion, presented by S. Dzhenzher, T. Garaev, O. Nikitenko, A. Petukhov, A. Skopenkov, A. Voropaev,
\url{https://www.mccme.ru/circles/oim/netflix.pdf} }

\newcommand{\dstt}{\bibitem[DS22]{DS22}  \emph{S. Dzhenzher and A. Skopenkov,} To the K\"uhnel conjecture on embeddability of $k$-complexes into $2k$-manifolds, arXiv:2208.04188.}

\newcommand{\embo}{\bibitem[Eb]{Eb} * \url{http://www.map.mpim-bonn.mpg.de/Embeddings_of_manifolds_with_boundary:_classification}}

\newcommand{\embe}{\bibitem[Em]{Em} * \url{http://www.map.mpim-bonn.mpg.de/Embedding_(simple_definition)}}

\newcommand{\ers}{\bibitem[ERS]{ERS} * Invariants of graph drawings in the plane, presented by A. Enne, A. Ryabichev, A. Skopenkov and T. Zaitsev, \url{http://www.turgor.ru/lktg/2017/6/index.htm}}


\newcommand{\feto}{\bibitem[Fe21]{Fe21} \emph{M. Fedorov.} A description of values of Seifert form for punctured $n$-manifolds in $(2n-1)$-space, arXiv:2107.02541.}

\newcommand{\ffen}{\bibitem[FF89]{FF89} * \emph{А. Т. Фоменко и Д. Б. Фукс.} Курс гомотопической топологии. М.: Наука, 1989.}

\newcommand{\ffene}{\bibitem[FF89]{FF89} * \emph{A.T. Fomenko and D.B. Fuchs.} Homotopical Topology, Springer, 2016.}


\newcommand{\fkosc}{\bibitem[FK17]{FK17} \emph{R. Fulek, J. Kyn{\v{c}}l,} Counterexample to an Extension of the Hanani-Tutte Theorem on the Surface of Genus 4, Combinatorica, 39 (2019) 1267--1279, arXiv:1709.00508.}

\newcommand{\fkos}{\bibitem[FK17]{FK17} \emph{R. Fulek, J. Kyn{\v{c}}l,} Hanani-Tutte for approximating maps of graphs, arXiv:1705.05243.}

\newcommand{\fkon}{\bibitem[FK19]{FK19} \emph{R. Fulek, J. Kyn{\v{c}}l,}
$\Z_2$-genus of graphs and minimum rank of partial symmetric matrices,
35th Intern. Symp. on Comp. Geom. (SoCG 2019), Article No. 39; pp. 39:1--39:16,
\url{https://drops.dagstuhl.de/opus/volltexte/2019/10443/pdf/LIPIcs-SoCG-2019-39.pdf}.
We refer to numbering in arXiv version: arXiv:1903.08637.}

\newcommand{\fktnf}{\bibitem[FKT]{FKT} \emph{M. H. Freedman, V. S. Krushkal and P. Teichner.} Van Kampen's
embedding obstruction is incomplete for 2-complexes in~$\R^4$, Math. Res. Letters. 1994. 1. P.~167-176.}

\newcommand{\fltf}{\bibitem[Fl34]{Fl34} \emph{A. Flores}, \"Uber $n$-dimensionale Komplexe die im $E^{2n+1}$ absolut selbstverschlungen sind, Ergeb. Math. Koll. 6 (1934) 4--7.}

\newcommand{\fo}{\bibitem[Fo]{Fo} * \emph{L. Fortnow.} Time for Computer Science to Grow Up,  \url{https://people.cs.uchicago.edu/~fortnow/papers/growup.pdf}.}

\newcommand{\fozf}{\bibitem[Fo04]{Fo04} * \emph{R. Fokkink.} A forgotten mathematician, Eur. Math. Soc. Newsletter 52 (2004) 9--14.}


\newcommand{\fpstz}{\bibitem[FPS]{FPS} \emph{R. Fulek, M.J. Pelsmajer and M. Schaefer.}
Strong Hanani-Tutte for the Torus, arXiv:2009.01683.}

\newcommand{\frse}{\bibitem[Fr78]{Fr78} \emph{M. Freedman,} Quadruple points of 3-manifolds in $S^4$, Comment. Math. Helv. 53 (1978), 385-394.}

\newcommand{\fres}{\bibitem[FR86]{FR86} \emph{R. Fenn, D. Rolfsen.}
Spheres may link homotopically in 4-space, J. London Math. Soc. 34 (1986) 177-184.}

\newcommand{\frof}{\bibitem[Fr15]{Fr15} \emph{F. Frick}, Counterexamples to the topological Tverberg conjecture,
Oberwolfach reports, 12:1 (2015), 318--321. arXiv:1502.00947.}


\newcommand{\fros}{\bibitem[Fr17]{Fr17} \emph{F. Frick}, O\lowercase{N AFFINE TVERBERG-TYPE RESULTS WITHOUT CONTINUOUS GENERALIZATION}, arXiv:1702.05466}


\newcommand{\fstz}{\bibitem[FS20]{FS20} \emph{F. Frick and P. Sober\'on}, The topological Tverberg problem beyond prime powers, arXiv:2005.05251.}

\newcommand{\fvto}{\bibitem[FV21]{FV21} \emph{M. Filakovsk\'y, L. Vok\v r\'inek.} Computing homotopy classes for diagrams, 	arXiv:2104.10152.}

\newcommand{\fwz}{\bibitem[FWZ]{FWZ} \emph{M. Filakovsk\'y, U. Wagner, S. Zhechev.} Embeddability of simplicial complexes is undecidable. Oberwolfach reports, to appear.}

\newcommand{\fwztz}{\bibitem[FWZ]{FWZ} \emph{M. Filakovsk\'y, U. Wagner, S. Zhechev.} Embeddability of simplicial complexes is undecidable. Proceedings of the 2020 ACM-SIAM Symposium on Discrete Algorithms.}



\newcommand{\ga}{\bibitem[GA]{GA} * \url{https://en.wikipedia.org/wiki/Galactic_algorithm}}

\newcommand{\gatt}{\bibitem[Ga22]{Ga22} T. Garaev, On winding numbers in $K_5$ minus an edge drawn in the plane, draft.}

\newcommand{\gdikrse}{\bibitem[GDI]{GDI} * {\it A. Chernov, A. Daynyak, A. Glibichuk, M. Ilyinskiy, A. Kupavskiy, A. Raigorodskiy and A. Skopenkov,} Elements of Discrete Mathematics As a Sequence of Problems (in Russian),
MCCME, Moscow, 2016. Update: \url{http://www.mccme.ru/circles/oim/discrbook.pdf}}

\newcommand{\gdikrs}{\bibitem[GDI]{GDI} * {\it А.А. Глибичук, А.Б. Дайняк, Д.Г. Ильинский, А.Б. Купавский, А.М. Райгородский, А.Б. Скопенков, А.А. Чернов,} Элементы дискретной математики в задачах, М, МЦНМО, 2016.
\url{http://www.mccme.ru/circles/oim/discrbook.pdf}}

\newcommand{\giso}{\bibitem[Gi71]{Gi71} * {\it S. Gitler,} Immersion and Embedding of Manifolds,
Proc. Symp. Pure Math. 22, 87-96 (1971).}

\newcommand{\gkp}{\bibitem[GKP]{GKP} * {\it R. Graham, D. Knuth, and O. Patashnik,} Concrete Mathematics: A Foundation for Computer Science, Addison–Wesley, first published in 1989, \url{https://www.csie.ntu.edu.tw/~r97002/temp/Concrete\%20Mathematics\%202e.pdf}.}

\newcommand{\gmpptw}{\bibitem[GMP+]{GMP+} \emph{X. Goaoc, I. Mabillard, P. Pat\'ak, Z. Pat\'akov\'a, M. Tancer, U. Wagner}, On Generalized Heawood Inequalities for Manifolds: a van Kampen--Flores-type Nonembeddability Result,
Israel J. Math., 222(2) (2017) 841-866. arXiv:1610.09063.}

\newcommand{\group}{\bibitem[Gr]{Gr} * \url{https://en.wikipedia.org/wiki/Groupthink}}

\newcommand{\grsz}{\bibitem[Gr69]{Gr69} \emph{B. Gr\"unbaum.} Imbeddings of simplicial complexes. Comment. Math. Helv., 44:1, 502--513, 1969.}


\newcommand{\gres}{\bibitem[Gr86]{Gr86} * \emph{M. Gromov}, Partial Differential Relations,
Ergebnisse der Mathematik und ihrer Grenzgebiete (3), Springer Verlag, Berlin-New York, 1986.}

\newcommand{\groz}{\bibitem[Gr10]{Gr10} \emph{M. Gromov,}
\newblock Singularities, expanders and topology of maps. Part 2: From combinatorics to topology via algebraic isoperimetry, \newblock Geometric and Functional Analysis 20 (2010), no.~2, 416--526.}

\newcommand{\grsn}{\bibitem[GR79]{GR79} \emph{J. L. Gross	and R. H. Rosen}, A linear time planarity algorithm for 2-complexes, Journal of the ACM, 26:4 (1979), 611--617.}

\newcommand{\gs}{\bibitem[GS]{GS} \emph{М. Гортинский и О. Скрябин.} Критерий вложимости графов в плоскость вдоль прямой, препринт.}

\newcommand{\gssn}{\bibitem[GS79]{GS} \emph{P.~M. Gruber and R.~Schneider,} Problems in geometric convexity. In {\em Contributions to geometry ({P}roc. {G}eom. {S}ympos., {S}iegen, 1978)}, 255--278. Birkh{\"a}user, Basel-Boston, Mass., 1979.}

\newcommand{\gsnn}{\bibitem[GS99]{GS99} \emph{R. Gompf and A. Stipsicz,}
4-manifolds and Kirby calculus, GSM20, AMS, Providence, RI, 1999.}


\newcommand{\gszs}{\bibitem[GS06]{GS06} \emph{D. Goncalves and A. Skopenkov,} Embeddings of homology equivalent manifolds with boundary, Topol. Appl., 153:12 (2006) 2026-2034. arxiv:1207.1326.}

\newcommand{\gssoe}{\bibitem[GSS+]{GSS+} * Projections of skew lines, presented by A. Gaifullin, A. Shapovalov, A. Skopenkov and M. Skopenkov, \url{http://www.turgor.ru/lktg/2001/index.php}.}

\newcommand{\gtes}{\bibitem[GT87]{GT87} * \emph{J. L. Gross and T. W. Tucker.}
Topological graph theory. New York: Wiley-Interscience, 1987.}

\newcommand{\guzn}{\bibitem[Gu09]{Gu09} \emph{A. Gundert.} On the complexity of embeddable simplicial complexes. Diplomarbeit, Freie Universit\"at Berlin, 2009. 	arXiv:1812.08447.}


\newcommand{\ha}{\bibitem[Ha]{Ha} * \emph{F. Harary.} Graph theory.
Рус. пер.: Ф. Харари. Теория графов. М., Мир, 1973.}

\newcommand{\hats}{\bibitem[Ha37]{Ha37} \emph{W. Hantzsche,} Einlagerung von Mannigfaltigkeiten in euklidische R\" aume, Math. Zeitschrift, 43:1 (1937) 38--58.}

\newcommand{\hastk}{\bibitem[Ha62k]{Ha62k} {\em A.~Haefliger,}  Knotted $(4k-1)$-spheres in $6k$-space, Ann. of Math. 75 (1962) 452--466.}

\newcommand{\hastl}{\bibitem[Ha62l]{Ha62l} \emph{A. Haefliger,} Differentiable links, Topology, 1 (1962) 241--244.}

\newcommand{\hast}{\bibitem[Ha63]{Ha63} \emph{A.~Haefliger,} Plongements differentiables dans le domain stable, Comment. Math. Helv. 36 (1962-63) 155--176.}

\newcommand{\hassa}{\bibitem[Ha66A]{Ha66A} \textit{A. Haefliger}. Differential embeddings of~$S^n$ in $S^{n+q}$ for $q>2$. Ann. Math. (2), 83 (1966), 402--~436.}

\newcommand{\hass}{\bibitem[Ha66C]{Ha66C} \emph{A.~Haefliger,}  Enlacements de spheres en codimension superiure a 2, Comment. Math. Helv. 41 (1966-67) 51--72.}

\newcommand{\hase}{\bibitem[Ha68]{Ha68} \emph{A. Haefliger,} Knotted Spheres and Related Geometric Topic,
in Proc. Int. Congr. Math., Moscow, 1966 (Mir, Moscow, 1968), 437--445.}

\newcommand{\hasn}{\bibitem[Ha69]{Ha69} \emph{L.~S.~Harris,} Intersections and embeddings of polyhedra, Topology 8 (1969) 1--26.}

\newcommand{\hasf}{\bibitem[Ha74]{Ha74} * \emph{P. Halmos,} How to talk mathematics. Notices of the Amer. Math. Soc., 21 (1974) 155--158.}

\newcommand{\haef}{\bibitem[Ha84]{Ha84} \emph{N. Habegger,} Obstruction to embedding disks II: a proof of a conjecture by Hudson, Topol. Appl. 17 (1984).}

\newcommand{\haes}{\bibitem[Ha86]{Ha86} \emph{N. Habegger,} Knots and links in codimension greater than 2, Topology, 25:3 (1986) 253--260.}

\newcommand{\hifn}{\bibitem[Hi59]{Hi59} \emph{M. W. Hirsch.} Immersions of manifolds, Trans. Amer. Math. Soc. 93 (1959) 242--276.}

\newcommand{\hjsf}{\bibitem[HJ64]{HJ64} \emph{R. Halin and H. A. Jung.}
Karakterisierung der Komplexe der Ebene und der 2-Sph\"are, Arch. Math. 1964. 15. P.~466-469.}

\newcommand{\hkne}{\bibitem[HK98]{HK98} \emph{N. Habegger and U. Kaiser,} Link homotopy in 2--metastable range, Topology 37:1 (1998) 75--94.}

\newcommand{\hmsnt}{\bibitem[HMS]{HMS93} * \emph{C. Hog-Angeloni, W. Metzler and A. J. Sieradski.}
Two-dimensional homotopy and combinatorial group theory. Cambridge: Cambridge Univ. Press, 1993. (London Math. Soc. Lecture Notes, 197).}

\newcommand{\ho}{\bibitem[Ho]{Ho} * The Hopf fibration, \url{https://www.youtube.com/watch?v=AKotMPGFJYk}}

\newcommand{\hozs}{\bibitem[Ho06]{Ho06} \emph{H. van der Holst,} Graphs and obstructions in four dimensions, J. Combin. Theory Ser. B 96:3 (2006), 388--404.}


\newcommand{\hpzn}{\bibitem[HP09]{HP09} \emph{H. van der Holst and R. Pendavingh,} On a graph property generalizing planarity and flatness, Combinatorica, 29 (2009) 337--361.}

\newcommand{\htsf}{\bibitem[HT74]{HT74} \emph{J. Hopcroft and R. E. Tarjan,} Efficient planarity testing, J. of the Association for Computing Machinery, 21:4 (1974) 549--568.}

\newcommand{\hufn}{\bibitem[Hu59]{hu59} * \emph{S. T. Hu,} Homotopy Theory, Academic Press, New York, 1959.}

\newcommand{\husn}{\bibitem[Hu69]{Hu69} * \emph{J. F. P. Hudson.} Piecewise linear topology, W. A. Benjamin, Inc., New York-Amsterdam, 1969.}


\newcommand{\io}{\bibitem[Io]{Io} * \url{https://en.wikipedia.org/wiki/Category:Impossible_objects}}

\newcommand{\info}{\bibitem[IF]{IF} * \url{http://www.map.mpim-bonn.mpg.de/Intersection_form}}

\newcommand{\irsf}{\bibitem[Ir65]{Ir65} \emph{M.~C.~Irwin,} Embeddings of polyhedral manifolds, Ann. of Math. (2)
82 (1965) 1--14.}

\newcommand{\isot}{\bibitem[Is]{Is} * \url{http://www.map.mpim-bonn.mpg.de/Isotopy}}


\newcommand{\jqnt}{\bibitem[JQ93]{JQ93} * \emph{A. Jaffe, F. Quinn,} ``Theoretical mathematics'': Toward a cultural synthesis of mathematics and theoretical physics. Bull.Am.Math.Soc. 29 (1993) 1-13. arXiv:math/9307227.}

\newcommand{\jozt}{\bibitem[Jo02]{Jo02} \emph{C. M. Johnson.} An obstruction to embedding a simplicial $n$-complex into a $2n$-manifold, Topology Appl. 122:3 (2002) 581--591.}

\newcommand{\jvz}{\bibitem[JVZ]{JVZ} D. Joji\'c, S. T. Vre\'cica, R. T. \v Zivaljevi\' c,
Topology and combinatorics of 'unavoidable complexes', arXiv:1603.08472v1.}


\newcommand{\kalai}{\bibitem[Ka]{Ka} G. Kalai, From Oberwolfach: The Topological Tverberg Conjecture is False, `Combinatorics and more' blog post, February 6, 2015, \url{gilkalai.wordpress.com}}

\newcommand{\kh}{\bibitem[Kh]{Kh} \emph{А.И. Храбров.} Руководство по чтению лекций
\url{http://vm.tstu.tver.ru/topics/pdf_tests/lection.pdf}}

\newcommand{\kho}{\bibitem[Kho]{Kho} \emph{N. Khoroshavkina.} A simple characterization of graphs of cutwidth 2, arXiv:1811.06716.}

\newcommand{\kkrot}{\bibitem[KKR]{KKR} \emph{K. Kawarabayashi, Y. Kobayashi and B. Reed.} The disjoint paths problem in quadratic time, J. of Comb. Theory, Ser. B, 102:2 (2012), 424--435.}

\newcommand{\kmsth}{\bibitem[KM63]{KM63} \emph{M. A. Kervaire and J. W. Milnor,} Groups of homotopy spheres. I,  Ann. of Math. (2) 77 (1963), 504-537.}

\newcommand{\kozeru}{\bibitem[Ko18]{Ko18} * \emph{Е. Колпаков.}
Доказательство теоремы Радона при помощи понижения размерности, Мат. Просвещение, 23 (2018), arXiv:1903.11055.}

\newcommand{\koze}{\bibitem[Ko18]{Ko18} * \emph{E. Kolpakov.}
A proof of Radon Theorem via lowering of dimension, Mat. Prosveschenie, 23 (2018), arXiv:1903.11055.}

\newcommand{\ko}{\bibitem[Ko]{Ko} \emph{E. Kolpakov.} A `converse' to the Constraint Lemma, arXiv:1903.08910.}

\newcommand{\koon}{\bibitem[Ko19]{Ko19} \emph{E. Kogan.} Linking of three triangles in 3-space, arXiv:1908.03865.}

\newcommand{\koto}{\bibitem[Ko21]{Ko21} \emph{E. Kogan.} On the rank of $\Z_2$-matrices with free entries on the diagonal, arXiv:2104.10668.}

\newcommand{\koee}{\bibitem[Ko88]{Ko88} \emph{U. Koschorke.} Link maps and the geometry of their invariants,
Manuscripta Math. 61:4 (1988) 383--415.}

\newcommand{\kps}{\bibitem[KPS]{KPS} * \emph{A. Kaibkhanov, D. Permyakov and A. Skopenkov.}
Realization of graphs with rotation, \url{http://www.turgor.ru/lktg/2005/3/index.htm}.}

\newcommand{\krzz}{\bibitem[Kr00]{Kr00} \emph{V. S. Krushkal.} Embedding obstructions and 4-dimensional thickenings of 2-complexes, Proc. Amer. Math. Soc. 128:12 (2000) 3683--3691. arXiv:math/0004058. }

\newcommand{\ksnn}{\bibitem[KS99]{KS99} * \emph{П. Кожевников и А. Скопенков.} Узкие деревья на плоскости, Мат. Образование. 1999. 2-3. С.~126-131.}

\newcommand{\kstz}{\bibitem[KS20]{KS20} \emph{R. Karasev and A. Skopenkov.}
Some `converses' to intrinsic linking theorems, arXiv:2008.02523.}

\newcommand{\ksto}{\bibitem[KS21]{KS21} * \emph{E. Kogan and A. Skopenkov.} A short exposition of the Patak-Tancer theorem on non-embeddability of $k$-complexes in $2k$-manifolds,  arXiv:2106.14010.}

\newcommand{\kstoe}{\bibitem[KS21e]{KS21e} \emph{E. Kogan and A. Skopenkov.}
Embeddings of $k$-complexes in $2k$-manifolds and minimum rank of partial symmetric matrices, arXiv:2112.06636.}

\newcommand{\kuse}{\bibitem[Ku68]{Ku68} * \emph{К. Куратовский.} Топология. Т.~1,~2. М.: Мир, 1969.}

\newcommand{\kunfo}{\bibitem[Ku94]{Ku94} \emph{W. K\"uhnel.} Manifolds in the skeletons of convex polytopes, tightness, and generalized Heawood inequalities. In Polytopes: abstract, convex and computational (Scarborough, ON, 1993), volume 440 of NATO Adv. Sci. Inst. Ser. C Math. Phys. Sci., pp. 241--247. Kluwer
Acad. Publ., Dordrecht, 1994.}


\newcommand{\kunf}{\bibitem[Ku95]{Ku95} * \emph{W. K\"uhnel}, Tight Polyhedral Submanifolds and Tight Triangulations, Lecture Notes in Math. 1612, Springer, 1995.}


\newcommand{\lazz}{\bibitem[La00]{La00} \emph{F. Lasheras.} An obstruction to 3-dimensional thickening,
Proc. Amer. Math. Soc. 2000. 128. P.~893-902.}

\newcommand{\lfma}{\bibitem[LF]{LF} \url{http://www.map.mpim-bonn.mpg.de/Linking_form}}

\newcommand{\lloe}{\bibitem[LL18]{LL18} \emph{A.S. Levine and T. Lidman.} Simply connected, spineless 4-manifolds, Forum of Math., Sigma, 7 (2019) e14, 1--11, arxiv:1803.01765.}

\newcommand{\lo}{\bibitem[Lo]{Lo} M.~de~Longueville. Notes on the topological Tverberg theorem.
Discrete Math.  247 (2002), no.~1--3, 271--297.
(The paper first appeared in
Discrete Math. 241 (2001) 207--233, but the original version suffered from serious publisher's typesetting errors.)}

\newcommand{\loot}{\bibitem[Lo13]{Lo13} \emph{M. de Longueville.} A course in topological combinatorics. Universitext. Springer, New York (2013).}

\newcommand{\lssn}{\bibitem[LS69]{LS69} \emph{W. B. R. Lickorish and L. C. Siebenmann.}
Regular neighborhoods and the stable range,  Trans. Amer. Math. Soc.. 1969. 139. P.~207-230.}

\newcommand{\lsne}{\bibitem[LS98]{LS98} \emph{L. Lovasz and A. Schrijver,}
A Borsuk theorem for antipodal links and a spectral characterization of linklessly embeddable graphs, Proc. Amer. Math. Soc. 126:5 (1998), 1275-1285.}

\newcommand{\ltof}{\bibitem[LT14]{LT14} \emph{E. Lindenstrauss and M. Tsukamoto,} Mean dimension and an embedding problem: an example, Israel J. Math. 199 (2014).}


\newcommand{\lyzf}{\bibitem[LY04]{LY04} * \emph{Y. Lin and A. Yang,} On 3-cutwidth critical graphs, Discrete Mathematics, 275 (2004), 339--346.}

\newcommand{\lz}{\bibitem[LZ]{LZ} * \emph{S. Lando and A. Zvonkin.} Embedded Graphs. Springer.}



\newcommand{\mast}{\bibitem[Ma73]{Ma73} \emph{С. В. Матвеев.} Специальные остовы кусочно-линейных многообразий, Мат. Сборник. 1973. 92. С.~282-293.}

\newcommand{\maste}{\bibitem[Ma73]{Ma73} \emph{S. V. Matveev.} Special skeletons of PL manifolds (in Russian), Mat. Sbornik. 1973. 92. P.~282-293.}

\newcommand{\mans}{\bibitem[Ma97]{Ma97} \emph{Yu. Makarychev.} A short proof of Kuratowski's graph planarity criterion, J. of Graph Theory, 25 (1997), 129--131.}

\newcommand{\mazt}{\bibitem[Ma03]{Ma03} * \emph{J.~Matou{\v{s}}ek.} Using the {B}orsuk-{U}lam theorem:
Lectures on topological methods in combinatorics and geometry. Springer Verlag, 2008.}


\newcommand{\mazf}{\bibitem[Ma05]{Ma05} \emph{V. Manturov.} A proof of the Vasiliev conjecture on the planarity of singular links, Izv. RAN 2005.}

\newcommand{\metn}{\bibitem[Me29]{Me29} \emph{K. Menger.} \"Uber pl\"attbare Dreiergraphen und Potenzen nicht pl\"attbarer Graphen, Ergebnisse Math. Kolloq., 2 (1929) 30--31.}

\newcommand{\mezf}{\bibitem[Me04]{Me04} \emph{S. Melikhov.} Sphere eversions and realization of mappings, Trudy MIAN 247 (2004) 159-181 (in Russian) arXiv:math.GT/0305158.}

\newcommand{\mezs}{\bibitem[Me06]{Me06} \emph{S. A. Melikhov}, The van Kampen obstruction and its relatives, 	
Proc. Steklov Inst. Math 266 (2009), 142-176 (= Trudy MIAN 266 (2009), 149-183), arXiv:math/0612082.}

\newcommand{\meoo}{\bibitem[Me11]{Me11} \emph{S. A. Melikhov}, Combinatorics of embeddings, arXiv:1103.5457.}

\newcommand{\meos}{\bibitem[Me17]{Me17} \emph{S. Melikhov,} Gauss type formulas for link map invariants, arXiv:1711.03530.}

\newcommand{\meoe}{\bibitem[Me18]{Me18} \emph{S. A. Melikhov,} A triple-point Whitney trick, J. Topol. Anal., 2018, 1--6.}


\newcommand{\miso}{\bibitem[Mi61]{Mi61} \emph{J. Milnor,} A procedure for killing homotopy groups of differentiable manifolds, Proc. Sympos. Pure Math, Vol. III (1961), 39--55.}

\newcommand{\mins}{\bibitem[Mi97]{Mi97} \emph{P. Minc.} Embedding simplicial arcs into the plane, Topol. Proc. 1997. 22. 305--340.}


\newcommand{\moss}{\bibitem[Mo77]{Mo77} * \emph{E. E. Moise.} Geometric Topology in Dimensions 2 and 3 (GTM), Springer-Verlag, 1977.}

\newcommand{\moen}{\bibitem[Mo89]{Mo89} \textit{B. Mohar}. An obstruction to embedding graphs in
surfaces. Discrete Math. 78 (1989) 135--142.}

\newcommand{\mrst}{\bibitem[MRS+]{MRS+} \emph{A. de Mesmay, Y. Rieck, E. Sedgwick, M. Tancer,}
Embeddability in $\R^3$ is NP-hard. arXiv:1708.07734.}

\newcommand{\mesczs}{\bibitem[MS06]{MS06} \emph{S.A. Melikhov, E.V. Shchepin,} The telescope approach to embeddability of compacta. arXiv:math.GT/0612085.}

\newcommand{\mstwof}{\bibitem[MST+]{MST+} \emph{J. Matou\v sek, E. Sedgwick, M. Tancer, U. Wagner}, Embeddability in the 3-sphere is decidable, Journal of the ACM 65:1 (2018) 1--49, arXiv:1402.0815.}


\newcommand{\mtzo}{\bibitem[MT01]{MT01} * \emph{B. Mohar and C. Thomassen.} Graphs on Surfaces.
The John Hopkins University Press, 2001.}

\newcommand{\mtwoz}{\bibitem[MTW10]{MTW10} \emph{J. Matou\v sek, M. Tancer, U. Wagner.} A geometric proof of
the colored Tverberg theorem, Discr. and Comp. Geometry, 47:2 (2012), 245--265. arXiv:1008.5275.}


\newcommand{\mtwoo}{\bibitem[MTW]{MTW} \emph{J. Matou\v sek, M. Tancer, U. Wagner.}
Hardness of embedding simplicial complexes in $\R^d$, J. Eur. Math. Soc. 13:2 (2011), 259--295. arXiv:0807.0336.}


\newcommand{\mwofo}{\bibitem[MW14]{MW14} \emph{I. Mabillard and U. Wagner.} Eliminating Tverberg Points, I. An Analogue of the Whitney Trick, Proc. of the 30th Annual Symp. on Comp. Geom. (SoCG'14), ACM, New York, 2014, pp. 171--180.}

\newcommand{\mwof}{\bibitem[MW15]{MW15} \emph{I. Mabillard and U. Wagner.}
Eliminating Higher-Multiplicity Intersections, I. A Whitney Trick for Tverberg-Type Problems. arXiv:1508.02349.}


\newcommand{\mwos}{\bibitem[MW16]{MW16} \emph{I. Mabillard and U. Wagner.} Eliminating Higher-Multiplicity Intersections, II. The Deleted Product Criterion in the $r$-Metastable Range. arXiv:1601.00876v2.}

\newcommand{\mwosd}{\bibitem[MW16']{MW16'} \emph{I. Mabillard and U. Wagner.} Eliminating Higher-Multiplicity Intersections, II. The Deleted Product Criterion in the r-Metastable Range,
Proceedings of the 32nd Annual Symposium on Computational Geometry (SoCG'16).}


\newcommand{\neno}{\bibitem[Ne91]{Ne91} \emph{S. Negami.} Ramsey theorems for knots, links and spatial graphs,
Trans. Amer. Math. Soc., 324 (1991), 527--541.}



\newcommand{\nkon}{\bibitem[NKS]{NKS} * \emph{L. T. Nguyen, J. Kim, B. Shim.}
Low-Rank Matrix Completion: A Contemporary Survey. arXiv:1907.11705.}

\newcommand{\noss}{\bibitem[No76]{No76} * \emph{С. П. Новиков.} Топология-1. М.: Наука, 1976. (Итоги науки и техники. ВИНИТИ. Современные проблемы математики. Основные направления, 12).}

\newcommand{\nszn}{\bibitem[NS09]{NS09} \emph{I. Novik and E. Swartz,} Socles of Buchsbaum modules, complexes and posets, Adv. Math. 222 (2009), 2059-2084.}

\newcommand{\nwns}{\bibitem[NW97]{NW97} \emph{A. Nabutovsky, S. Weinberger}. Algorithmic aspects of homeomorphism problems. arXiv:math/9707232.}


\newcommand{\omoe}{\bibitem[Om18]{Om18} * \emph{А. Омельченко,} Теория графов. М.: МЦНМО, 2018.}

\newcommand{\orszo}{\bibitem[ORS]{ORS} \emph{A. Onischenko, D. Repov\v s and A. Skopenkov.}
Resolutions of 2-polyhedra by fake surfaces and embeddings into $\R^4$, Contemp. Math.  288 (2001) 396--400.}

\newcommand{\ossf}{\bibitem[OS74]{OS74} \emph{R. P. Osborne and R. S. Stevens.} Group presentations
corresponding to spines of 3-manifolds, I, Amer. J.~Math. 1974. 96. P.~454-471; II, Amer. J.~Math. 1977. 234.
P.~213-243; III, Amer. J.~Math. 1977. 234 P.~245-251.}


\newcommand{\oz}{\bibitem[Oz]{Oz} \emph{M. \"Ozaydin,} Equivariant maps for the symmetric group, unpublished,
\url{http://minds.wisconsin.edu/handle/1793/63829}.}

\newcommand{\panof}{\bibitem[Pan15]{Pan15} \emph{K. Panagiotis.} A note on the topology of irreducible $SO(3)$-manifolds, 	arXiv:1508.06150.}

\newcommand{\paof}{\bibitem[Pa15]{Pa15} \emph{S. Parsa,} On links of vertices in simplicial $d$-complexes embeddable in the euclidean $2d$-space, Discrete Comput. Geom. 59:3 (2018), 663--679.
This is arXiv:1512.05164v4 up to numbering of sections, theorems etc; we refer to numbering in arxiv version.
Correction: Discrete Comput. Geom. 64:3 (2020) 227--228.}

\newcommand{\paoe}{\bibitem[Pa18]{Pa18} \emph{S. Parsa,} On links of vertices in simplicial $d$-complexes
embeddable in the euclidean $2d$-space, arXiv:1512.05164v6.}

\newcommand{\patz}{\bibitem[Pa20]{Pa20} \emph{S. Parsa,} On links of vertices in simplicial $d$-complexes
embeddable in the euclidean $2d$-space, arXiv:1512.05164v8.}


\newcommand{\patzl}{\bibitem[Pa20]{Pa20} \emph{S. Parsa,}
Correction to: On the Links of Vertices in Simplicial $d$-Complexes Embeddable in the Euclidean $2d$-Space,
Discrete Comput. Geom. 64:3 (2020) 227--228.}

\newcommand{\patza}{\bibitem[Pa20]{Pa20} \emph{S. Parsa,} On the Smith classes, the van Kampen obstruction and embeddability of $[3]*K$, arXiv:2001.06478.}

\newcommand{\patzb}{\bibitem[Pa20b]{Pa20b} \emph{S. Parsa,} On the embeddability of $[3]*K$, arXiv:2001.06506.}

\newcommand{\pak}{\bibitem[Pa]{Pa} * \emph{I. Pak}, Lectures on Discrete and Polyhedral Geometry, \url{http://www.math.ucla.edu/~pak/geompol8.pdf}.}

\newcommand{\peze}{\bibitem[Pe08]{Pe08} \emph{Д. Пермяков.} Классификация погружений графов в плоскость,
Вестник МГУ, сер.1, 2008, N5, 55-56.}

\newcommand{\peos}{\bibitem[Pe16]{Pe16} \emph{Д. Пермяков.} Матем. сб., 207:6 (2016),  93--112.}

\newcommand{\pest}{\bibitem[Pe72]{Pe72} * \emph{B. B. Peterson.} The Geometry of Radon's Theorem, Amer. Math. Monthly 79 (1972), 949-963.}


\newcommand{\prnf}{\bibitem[Pr95]{Pr95} * \emph{V. V. Prasolov.} Intuitive topology. Amer. Math. Soc., Providence, R.I., 1995.}

\newcommand{\prnfr}{\bibitem[Pr95]{Pr95} * \emph{В. В. Прасолов.} Наглядная топология. М.: МЦНМО, 1995.}


\newcommand{\przs}{\bibitem[Pr06]{Pr06} * \emph{V. V. Prasolov.}
Elements of Combinatorial and Differential Topology, 2006, GSM 74, Amer. Math. Soc., Providence, RI.}

\newcommand{\przsru}{\bibitem[Pr04]{Pr04} * \emph{В. В. Прасолов.}
Элементы комбинаторной и дифференциальной топологии. М.: МЦНМО, 2004. \url{http://www.mccme.ru/prasolov}.}

\newcommand{\przse}{\bibitem[Pr07]{Pr07} * \emph{V. V. Prasolov.} Elements of homology theory. 2007, GSM 74, Amer. Math. Soc., Providence, RI.}


\newcommand{\przseru}{\bibitem[Pr06]{Pr06} * \emph{В. В. Прасолов.} Элементы теории гомологий. М.: МЦНМО, 2006.}


\newcommand{\psns}{\bibitem[PS96]{PS96} * \emph{V. V. Prasolov, A. B. Sossinsky } Knots, Links, Braids, and 3-manifolds. Amer. Math. Soc. Publ., Providence, R.I., 1996.}


\newcommand{\pszf}{\bibitem[PS05]{PS05} * \emph{В. В. Прасолов и М. Б. Скопенков.}
Рамсеевская теория зацеплений, Мат. Просвещение. 2005. 9. С.~108--115.}

\newcommand{\pszfen}{\bibitem[PS05]{PS05} * \emph{V. V. Prasolov and M.B. Skopenkov.}
Ramsey link theory, Mat, Prosvescheniye, 9 (2005), 108--115.}

\newcommand{\psoo}{\bibitem[PS11]{PS11} \emph{Y. Ponty and C. Saule.} A combinatorial framework for designing (pseudoknotted) RNA algorithms, Proc. of the 11th Intern. Workshop on Algorithms in Bioinformatics, WABI'11, 250--269.}


\newcommand{\pstz}{\bibitem[PS20]{PS20} \emph{S. Parsa and A. Skopenkov.} On embeddability of joins and their `factors', arXiv:2003.12285.}


\newcommand{\psszn}{\bibitem[PSS]{PSS} \emph{M. J. Pelsmajer, M. Schaefer and D. Stasi.} Strong Hanani-Tutte on the projective plane. SIAM J. Discrete Math., 23:3 (2009) 1317--1323.}

\newcommand{\pton}{\bibitem[PT19]{PT19} \emph{P. Pat\'ak and M. Tancer.} Embeddings of $k$-complexes into $2k$-manifolds. arXiv:1904.02404.}

\newcommand{\pw}{\bibitem[PW]{PW} \emph{I. Pak, S. Wilson}, G\lowercase{EOMETRIC REALIZATIONS OF POLYHEDRAL COMPLEXES}, \url{http://www.math.ucla.edu/~pak/papers/Fary-full31.pdf}.}


\newcommand{\razf}{\bibitem[RA05]{RA05} * \emph{J. L. Ram\'irez Alfons\'in.} Knots and links in spatial graphs: a survey. Discrete Math., 302 (2005), 225--242.}

\newcommand{\rep}{\bibitem[Rep]{Rep} Referee's report on the paper ``Some `converses' to intrinsic linking theorems', \url{https://www.mccme.ru/circles/oim/materials/ksreport.pdf}}

\newcommand{\rnoo}{\bibitem[RN11]{RN11} * \emph{R. L. Ricca, B. Nipoti.} Gauss' linking number revisited.
J. of Knot Theory and Its Ramif. 20:10 (2011) 1325--1343. \url{https://www.maths.ed.ac.uk/~v1ranick/papers/ricca.pdf} .}

\newcommand{\rrstz}{\bibitem[RRS]{RRS} * \emph{V. Retinskiy, A. Ryabichev and A. Skopenkov.}
Motivated exposition of the proof of the Tverberg Theorem (in Russian).
Mat. Prosveschenie, 27 (2021), 166--169. arXiv:2008.08361.}

\newcommand{\rssec}{\bibitem[RS68]{RS68} \emph{C. P. Rourke and B. J. Sanderson,} Block bundles II, Ann. of Math. (2), 87 (1968) 431--483.}

\newcommand{\rsst}{\bibitem[RS72]{RS72} * \emph{C. P. Rourke and B. J. Sanderson,}
\newblock Introduction to Piecewise-Linear Topology,
\newblock \emph{Ergebn.\ der Math.} 69, Springer-Verlag, Berlin, 1972.}

\newcommand{\rsstr}{\bibitem[RS72]{RS72} * \emph{К. П. Рурк и Б. Дж. Сандерсон.} Введение в кусочно-линейную топологию, Москва. Мир. 1974.}

\newcommand{\rsns}{\bibitem[RS96]{RS96} * \emph{D. Repov\v s and A. B. Skopenkov.}
Embeddability and isotopy of polyhedra in Euclidean spaces,
Proc. of the Steklov Inst. Math. 1996. 212. P.~173-188.}

\newcommand{\rsne}{\bibitem[RS98]{RS98} \emph{D. Repov\v s and A. B. Skopenkov.}
A deleted product criterion for approximability of a map by embeddings, Topol. Appl. 1998. 87 P.~1-19.}

\newcommand{\rsnn}{\bibitem[RS99]{RS99} * \emph{D. Repov\v s and A. B. Skopenkov.} New results on embeddings of polyhedra and manifolds into Euclidean spaces,
Russ. Math. Surv. 54:6 (1999), 1149--1196.}


\newcommand{\rsnnd}{\bibitem[RS99']{RS99'} * \emph{Д. Реповш и А. Скопенков.}
Кольца Борромео и препятствия к вложимости, Труды МИРАН. 1999. 225. С.~331-338.}

\newcommand{\rszz}{\bibitem[RS00]{RS00} \emph{D. Repov\v s and A. Skopenkov.} Cell-like resolutions of polyhedra by special ones,  Colloq. Math. 2000. 86:2. P. 231--237.}

\newcommand{\rszzd}{\bibitem[RS00']{RS00'} * \emph{Д. Реповш и А. Скопенков.} Характеристические классы для начинающих, Мат. Просвещение. 2000. 4. С.~151-176.}

\newcommand{\rszo}{\bibitem[RS01]{RS01} \emph{D. Repovs and A. Skopenkov.} On contractible $n$-dimensional compacta, non-embeddable into $\R^{2n}$, Proc. Amer. Math. Soc. 129 (2001) 627--628.}

\newcommand{\rszt}{\bibitem[RS02]{RS02} * \emph{Д. Реповш и А. Скопенков.} Теория препятствий для начинающих,
Мат. Просвещение. 2002. 6. C.~60-77.}

\newcommand{\rszf}{\bibitem[RS04]{RS04} \emph{N. Robertson and P. Seymour.} Graph Minors. XX. Wagner's conjecture, J. of Comb. Theory, B, 92:2 (2004) 325--357.}

\newcommand{\rssnf}{\bibitem[RSS]{RSS95} \emph{D. Repov\v s, A. B. Skopenkov  and E. V. \v S\v cepin.}
On uncountable collections of continua and their span, Colloq. Math. 1995. 69:2. P.~289-296.}

\newcommand{\rssnfd}{\bibitem[RSS']{RSS95'} \emph{D. Repov\v s, A. B. Skopenkov and E. V \v S\v cepin.}
On embeddability of $X\times I$ into Euclidean space, Houston J.~Math. 1995. 21. P.~199-204.}

\newcommand{\rssz}{\bibitem[RSS+]{RSSZ} * \emph{A. Rukhovich, A. Skopenkov, M. Skopenkov, A. Zimin},
Realizability of hypergraphs, \url{https://www.turgor.ru/lktg/2013/1/1-1en.pdf} .}


\newcommand{\rstnt}{\bibitem[RST']{RST93} \emph{N. Robertson, P. Seymour and R. Thomas}, Linkless embeddings of graphs in 3-space, Bull. of the Amer. Math. Soc., 21 (1993) 84--89.}

\newcommand{\rstno}{\bibitem[RST]{RST91} * \emph{N. Robertson, P. Seymour and R. Thomas}, A survey of
linkless embeddings, Graph Structure Theory (Seattle, WA, 1991), Contemp. Math. 147, (1993) 125--136.}


\newcommand{\rwzl}{\bibitem[RWZ+]{RWZ+} \emph{Y. Ren, C. Wen, S. Zhen, N. Lei, F. Luo, D.X. Gu},
Characteristic class of isotopy for surfaces, J. Syst. Sci. Complex. 33 (2020) 2139--2156.}


\newcommand{\saeo}{\bibitem[Sa81]{Sa81} \emph{H. Sachs.} On spatial representation of finite graphs,
in: Finite and infinite sets (Eger, 1981), 649--662, Colloq. Math. Soc. Janos Bolyai, 37, North-Holland, Amsterdam, 1984.}

\newcommand{\sano}{\bibitem[Sa91]{Sa91} \emph{K. S. Sarkaria.}
A one-dimensional Whitney trick and Kuratowski's graph planarity criterion, Israel J.~Math. 73 (1991), 79--89.
\url{http://kssarkaria.org/docs/One-dimensional.pdf}.}

\newcommand{\sanov}{\bibitem[Sa91g]{Sa91g} \emph{K. S. Sarkaria.} A generalized Van Kampen-Flores theorem, Proc. Amer. Math. Soc. 111 (1991), 559--565.}

\newcommand{\sant}{\bibitem[Sa92]{Sa92} \emph{K. S. Sarkaria.} Tverberg’s theorem via number fields. Israel J. Math., 79:317–320, 1992.}

\newcommand{\sann}{\bibitem[Sa99]{Sa99} O. Saeki {\em On punctured 3-manifolds in 5-sphere}, Hiroshima Math. J. 29 (1999) 255--272.}

\newcommand{\sazz}{\bibitem[Sa00]{Sa00} \emph{K. S. Sarkaria.} Tverberg partitions and Borsuk-Ulam theorems. Pacific J. Math., 196:1 (2000) 231--241.}

\newcommand{\sczf}{\bibitem[Sc04]{Sc04} \emph{T. Sch\"oneborn.} On the Topological Tverberg Theorem, arXiv:math/0405393.}


\newcommand{\scot}{\bibitem[Sc13]{Sc13} * \emph{M. Schaefer.} Hanani-Tutte and related results. In Geometry --- intuitive, discrete, and convex, Bolyai Soc. Math. Stud., 24 (2013), 259--299.
\url{http://ovid.cs.depaul.edu/documents/htsurvey.pdf} }


\newcommand{\sctz}{\bibitem[Sc20]{Sc20} \emph{M. Schaefer.} The Graph Crossing Number and
its Variants: A Survey. The Electr. J. of Comb. (2020), DS21, \url{https://www.combinatorics.org/files/Surveys/ds21/ds21v5-2020.pdf}}


\newcommand{\scef}{\bibitem[Sc84]{Sc84} \emph{E.~V.~\v S\v cepin.} Soft mappings of manifolds, Russian Math. Surveys, 39:5 (1984).}

\newcommand{\shfs}{\bibitem[Sh57]{Sh57} \emph{A. Shapiro,} Obstructions to the embedding of a complex in a Euclidean space, I, The first obstruction, Ann. Math. 66 (1957), 256--269.}


\newcommand{\shen}{\bibitem[Sh89]{Sh89} * \emph{Ю. А. Шашкин,} Неподвижные точки, М., Наука, 1989.}

\newcommand{\shoe}{\bibitem[Sh18]{Sh18} * \emph{S. Shlosman},  Topological Tverberg Theorem: the proofs and the counterexamples, Russian Math. Surveys, 73:2 (2018), 175–182. arXiv:1804.03120.}

\newcommand{\sisn}{\bibitem[Si69]{Si69} \emph{K. Sieklucki.} Realization of mappings, Fund. Math. 1969. 65. P.~325-343.}

\newcommand{\sios}{\bibitem[Si16]{Si16} \emph{S. Simon,} Average-Value Tverberg Partitions via Finite Fourier Analysis, Israel J. Math., 216 (2016), 891-904, arXiv:1501.04612.}



\newcommand{\sknf}{\bibitem[Sk94]{Sk94} \emph{А. Скопенков.} Геометрическое доказательство теоремы
Нойвирта об утолщаемости 2-мерных полиэдров, Math. Notes. 1995. 58:5. P.~1244-1247.}


\newcommand{\skne}{\bibitem[Sk98]{Sk98} \emph{A. B. Skopenkov.} On the deleted product criterion for embeddability in $\R^m$, Proc. Amer. Math. Soc. 1998. 126:8. P.~2467-2476.}

\newcommand{\skzz}{\bibitem[Sk00]{Sk00} \emph{A. Skopenkov,} On the generalized Massey--Rolfsen invariant for link maps, Fund. Math. 165 (2000), 1--15.}

\newcommand{\skzt}{\bibitem[Sk02]{Sk02} \emph{A. Skopenkov,} On the Haefliger-Hirsch-Wu invariants for embeddings and immersions, Comment. Math. Helv. 77 (2002), 78--124.}

\newcommand{\skzth}{\bibitem[Sk03]{Sk03} \emph{M. Skopenkov,} Embedding products of graphs into Euclidean spaces,
Fund. Math. 179 (2003),~191--198, arXiv:0808.1199.}

\newcommand{\skzthd}{\bibitem[Sk03']{Sk03'} \emph{M. Skopenkov,} On approximability by embeddings of cycles in the plane, Topol. Appl. 134 (2003),~1--22, arXiv:0808.1187.}

\newcommand{\skzf}{\bibitem[Sk05]{Sk05} * \emph{A. Skopenkov,}
On the Kuratowski graph planarity criterion, Mat. Prosveschenie, 9 (2005), 116-128. arXiv:0802.3820.}


\newcommand{\skzs}{\bibitem[Sk05i]{Sk05i} \emph{A. Skopenkov,} A new invariant and parametric connected sum of embeddings, Fund. Math. 197 (2007) 253--269. arxiv:math/0509621.}

\newcommand{\skzei}{\bibitem[Sk05]{Sk05} \emph{A.  Skopenkov,} A classification of smooth embeddings of
4-manifolds in 7-space, I, Topol. Appl., 157 (2010) 2094--2110. arXiv:math/0512594.}

\newcommand{\skze}{\bibitem[Sk06]{Sk06} * \emph{A. Skopenkov,} Embedding and knotting of manifolds in Euclidean spaces, London Math. Soc. Lect. Notes, 347 (2008) 248--342. arXiv:math/0604045.}

\newcommand{\skzsi}{\bibitem[Sk06']{Sk06'} \emph{A. Skopenkov,} A classification of smooth embeddings of 3-manifolds in 6-space, Math. Zeitschrift, 260:3 (2008) 647--672. arxiv:math/0603429.}

\newcommand{\skozp}{\bibitem[Sk08]{Sk08} \emph{A.  Skopenkov,} Embeddings of $k$-connected $n$-manifolds into
$\R^{2n-k-1}$. arxiv:math/0812.0263; earlier version published in Proc. Amer. Math. Soc., 138 (2010) 3377--3389.}

\newcommand{\skoz}{\bibitem[Sk10]{Sk10} * \emph{А. Скопенков,} Вложения в плоскость графов с вершинами степени 4,
Мат. Просвещение, 21 (2017), arXiv:1008.4940.}

\newcommand{\skoo}{\bibitem[Sk11]{Sk11} \emph{M. Skopenkov,} When is the set of embeddings finite up to isotopy? Intern. J. Math. 26:7 (2015), 28 pp. arXiv:1106.1878.}

\newcommand{\sks}{\bibitem[Sk14]{Sk14} * \emph{A. Skopenkov,} Realizability of hypergraphs and Ramsey link theory, arXiv:1402.0658.}


\newcommand{\skof}{\bibitem[Sk15]{Sk15} * \emph{А. Скопенков,} Алгебраическая топология с геометрической точки зрения, Москва, МЦНМО, 2015 (1е издание).}

\newcommand{\skofe}{\bibitem[Sk15]{Sk15} * \emph{A. Skopenkov,} Algebraic Topology From Geometric Viewpoint (in Russian), MCCME, Moscow, 2015 (1st edition). }

\newcommand{\skofel}{\bibitem[Sk15e]{Sk15e} * \emph{А. Скопенков,} Алгебраическая топология
с геометрической точки зрения, эл. версия, \url{http://www.mccme.ru/circles/oim/home/combtop13.htm\#photo}}


\newcommand{\skotzr}{\bibitem[Sk20]{Sk20} * \emph{А. Скопенков,} Алгебраическая топология с геометрической точки зрения, Москва, МЦНМО, 2020 (2е издание).
Часть книги: \url{http://www.mccme.ru/circles/oim/obstruct.pdf}}

\newcommand{\skotz}{\bibitem[Sk20]{Sk20} * \emph{A. Skopenkov,} Algebraic Topology From Geometric Viewpoint (in Russian), MCCME, Moscow, 2020 (2nd edition).
Part of the book: \url{http://www.mccme.ru/circles/oim/obstruct.pdf} .
Part of the English translation: \url{https://www.mccme.ru/circles/oim/obstructeng.pdf}. }



\newcommand{\skofp}{\bibitem[Sk15]{Sk15} \emph{A. Skopenkov,} Classification of knotted tori,
Proc. A of the Royal Soc. of Edinburgh, 150:2 (2020), 549-567. Full version: arXiv:1502.04470.}


\newcommand{\skos}{\bibitem[Sk16]{Sk16} * \emph{A. Skopenkov,} A user's guide to the topological Tverberg Conjecture, arXiv:1605.05141v4. Abridged earlier published version: Russian Math. Surveys, 73:2 (2018), 323--353.}



\newcommand{\skosd}{\bibitem[Sk16']{Sk16'} * \emph{A. Skopenkov,} Stability of intersections of graphs in the plane and the van Kampen obstruction, Topol. Appl. 240(2018) 259--269, arXiv:1609.03727.}


\newcommand{\skosc}{\bibitem[Sk16c]{Sk16c} * \emph{A. Skopenkov,}  Embeddings in Euclidean space: an introduction to their classification, to appear in Boll. Man. Atl. http://www.map.mpim-bonn.mpg.de/Embeddings\_in\_Euclidean\_space:\_an\_introduction\_to\_their\_classification}

\newcommand{\skosie}{\bibitem[Sk16e]{Sk16e} * \emph{A. Skopenkov,} Embeddings just below the stable range: classification, to appear in Boll. Man. Atl.
http://www.map.mpim-bonn.mpg.de/Embeddings\_just\_below\_the\_stable\_range:\_classification}

\newcommand{\skost}{\bibitem[Sk16t]{Sk16t} * \emph{A. Skopenkov,} 3-manifolds in 6-space, to appear in Boll. Man. Atl.
http://www.map.mpim-bonn.mpg.de/3-manifolds\_in\_6-space}

\newcommand{\skosf}{\bibitem[Sk16f]{Sk16f} * \emph{A. Skopenkov,} 4-manifolds in 7-space, to appear in Boll. Man. Atl. http://www.map.mpim-bonn.mpg.de/4-manifolds\_in\_7-space}

\newcommand{\skosh}{\bibitem[Sk16h]{Sk16h} * \emph{A. Skopenkov,} High codimension links, to appear in Boll. Man. Atl. \url{http://www.map.mpim-bonn.mpg.de/High_codimension_links}.}

\newcommand{\skosi}{\bibitem[Sk16i]{Sk16i} * \emph{A. Skopenkov,} Isotopy, submitted to Boll. Man. Atl.
\url{http://www.map.mpim-bonn.mpg.de/Isotopy}.}

\newcommand{\skose}{\bibitem[Sk17]{Sk17} \emph{A. Skopenkov,}
Eliminating higher-multiplicity intersections in the metastable dimension range. arXiv:1704.00143.}

\newcommand{\skosed}{\bibitem[Sk17v]{Sk17v} * \emph{A. Skopenkov,}
On van Kampen-Flores, Conway-Gordon-Sachs and Radon theorems,  arXiv:1704.00300.}

\newcommand{\sk}{\bibitem[Sk17o]{Sk17o} \emph{A. Skopenkov,} On the metastable Mabillard-Wagner conjecture.  arXiv:1702.04259.}

\newcommand{\skmos}{\bibitem[Sk17d]{Sk17d} \emph{M. Skopenkov}. Discrete field theory: symmetries and conservation laws, arXiv:1709.04788.}

\newcommand{\skoe}{\bibitem[Sk18]{Sk18} * \emph{A. Skopenkov.} Invariants of graph drawings in the plane.
Arnold Math. J., 6 (2020) 21--55; full version: arXiv:1805.10237.}


\newcommand{\skoeo}{\bibitem[Sk18o]{Sk18o} * \emph{A. Skopenkov.} A short exposition of S. Parsa's theorems on intrinsic linking and non-realizability. Discr. Comp. Geom. 65:2 (2021), 584--585; full version:  arXiv:1808.08363.}


\newcommand{\skona}{\bibitem[Sk19]{Sk19} * \emph{A. Skopenkov,} A short exposition of the Levine-Lidman example of spineless 4-manifolds, arXiv:1911.07330.}

\newcommand{\sktze}{\bibitem[Sk21m]{Sk21m} * \emph{A. Skopenkov.} Mathematics via Problems. Part 1: Algebra. Amer. Math. Soc., Providence, 2021. Preliminary version: \url{https://www.mccme.ru/circles/oim/algebra_eng.pdf}}

\newcommand{\sktz}{\bibitem[Sk20u]{Sk20u} * \emph{A. Skopenkov.} A user's guide to basic knot and link theory,
in Topology, Geometry, and Dynamics, Contemporary Mathematics, vol. 772, Amer. Math. Soc., Providence, RI, 2021, pp. 281--309.
Russian version: Mat. Prosveschenie 27 (2021), 128--165. arXiv:2001.01472.}

\newcommand{\sktzo}{\bibitem[Sk20o]{Sk20o} \emph{A. Skopenkov.} On some results of S. Abramyan and T. Panov, arXiv:2005.11152.}

\newcommand{\sktzr}{\bibitem[Sk20e]{Sk20e} * \emph{A. Skopenkov.}
Extendability of simplicial maps is undecidable, Discr. Comp. Geom., to appear,	arXiv:2008.00492.}

\newcommand{\sktzd}{\bibitem[Sk21d]{Sk21d} * \emph{A. Skopenkov.}
On different reliability standards in current mathematical research, arXiv:2101.03745.
More often updated version: \url{https://www.mccme.ru/circles/oim/rese_inte.pdf}.}

\newcommand{\sktt}{\bibitem[Sk22]{Sk22} * \emph{A. Skopenkov.} Invariants of embeddings of 2-surfaces in 3-space,
arXiv:2201.10944.}

\newcommand{\skd}{\bibitem[Sk]{Sk} * \emph{А. Скопенков.} Алгебраическая топология с алгоритмической точки зрения, 
\url{http://www.mccme.ru/circles/oim/algor.pdf}.}

\newcommand{\skde}{\bibitem[Sk]{Sk} * \emph{A. Skopenkov.} Algebraic Topology From Algorithmic Standpoint, draft of a book, mostly in Russian,
\url{http://www.mccme.ru/circles/oim/algor.pdf}.}


\newcommand{\skon}{\bibitem[Skw]{Skw} * \emph{A. Skopenkov.} Whitney trick for eliminating multiple intersections, slides for talks at St Petersburg, Brno, Kiev, Moscow,  \url{https://www.mccme.ru/circles/oim/eliminat_talk.pdf}.}

\newcommand{\skl}{\bibitem[EEF]{EEF} * {\it Proposed by D. Eliseev, A. Enne, M. Fedorov, A. Glebov, N. Khoroshavkina, E. Morozov, A. Skopenkov, R. \v Zivaljevi\'c.}
A user's guide to knot and link theory, \url{https://www.turgor.ru/lktg/2019/3} .}

\newcommand{\skr}{\bibitem[Skr]{Skr} * \emph{A. Skopenkov.} Realizability of hypergraphs, slides for talks,  \url{https://www.mccme.ru/circles/oim/algor1_beamer.pdf}.}

\newcommand{\skt}{\bibitem[Skt]{Skt} * \emph{A. Skopenkov.} Transparent anonymous peer review,
\linebreak
\url{https://www.mccme.ru/circles/oim/home/transp_peer_review.htm} .}

\newcommand{\rslktg}{\bibitem[KRR+]{RRSl} * Towards higher-dimensional combinatorial geometry, presented by
E. Kogan, V. Retinskiy, E. Riabov and A. Skopenkov, \url{https://www.mccme.ru/circles/oim/multicomb.pdf} .}


\newcommand{\sm}{\bibitem[Sm]{Sm} S. Smirnov.}

\newcommand{\sper}{\bibitem[Sp]{Sp} * Sperner's lemma defeats the rental harmony problem, \url{https://www.youtube.com/watch?v=7s-YM-kcKME}.}

\newcommand{\sset}{\bibitem[SS83]{SS83} \emph{Е. В. Щепин, М. А. Штанько.} Спектральный критерий вложимости компактов в евклидовы пространства, Труды Ленинградской Международной Топологической конференции. Л.: Наука, 1983. С.~135-142.}

\newcommand{\ssnt}{\bibitem[SS92]{SS92} \emph{J.~Segal and S.~Spie\.z.} Quasi embeddings and embeddings of polyhedra in $\R^m$,  Topol. Appl., 45 (1992) 275--282.}

\newcommand{\sszt}{\bibitem[SS03]{SS03} \emph{F. W. Simmons and F. E. Su.}
Consensus-halving via theorems of Borsuk-Ulam and Tucker, Math. Social Sciences 45 (2003) 15–25. \url{https://www.math.hmc.edu/~su/papers.dir/tucker.pdf}.}

\newcommand{\ssot}{\bibitem[SS13]{SS13} \emph{M. Schaefer and D. Stefankovi\v c.} Block additivity of $\Z_2$-embeddings. In Graph drawing, volume 8242 of Lecture Notes in Comput. Sci., 185--195.
Springer, Cham, 2013. \url{http://ovid.cs.depaul.edu/documents/genus.pdf}}

\newcommand{\sssne}{\bibitem[SSS]{SSS} \emph{J. Segal, A. Skopenkov and S. Spie\. z.}
Embeddings of polyhedra in $\R^m$ and the deleted product obstruction, Topol. Appl. 1998. 85. P.~225-234.}

\newcommand{\sstnf}{\bibitem[SST95]{SST95} \emph{R. S. Simon, S. Spie\. z and H. Toru\'nczyk.}
T\lowercase{HE EXISTENCE OF EQUILIBRIA IN CERTAIN GAMES, SEPARATION FOR FAMILIES OF CONVEX FUNCTIONS
AND A THEOREM OF BORSUK-ULAM TYPE}, Israel J. Math 92 (1995) 1--21.}

\newcommand{\sstzt}{\bibitem[SST02]{SST02} \emph{R. S. Simon, S. Spie\. z and H. Toru\'nczyk.}
E\lowercase{QUILIBRIUM EXISTENCE AND TOPOLOGY IN SOME REPEATED GAMES WITH INCOMPLETE INFORMATION},
Trans. Amer. Math. Soc. 354:12 (2002) 5005-5026.}

\newcommand{\stez}{\bibitem[ST80]{ST80} * {\it H.~Seifert and W.~Threlfall.}
A textbook of topology, v~89 of {\em Pure and Applied Mathematics}.
Academic Press, New York-London, 1980.}


\newcommand{\stzs}{\bibitem[ST07]{ST07} * \emph{А. Скопенков и А. Телишев.}
И вновь о критерии Куратовского планарности графов, Мат. Просвещение, 11 (2007), 159--160.}

\newcommand{\stzse}{\bibitem[ST07]{ST07} * \emph{A. Skopenkov and A. Telishev}, Once again on the Kuratowski graph planarity criterion, Mat. Prosveschenie, 11 (2007), 159-160. arXiv:0802.3820.}

\newcommand{\stos}{\bibitem[ST17]{ST17} \emph{A. Skopenkov  and M. Tancer,}
Hardness of almost embedding simplicial complexes in $\R^d$, Discr. Comp. Geom., 61:2 (2019), 452--463. arXiv:1703.06305.}

\newcommand{\stwh}{\bibitem[SW]{SW} * \url{http://www.map.mpim-bonn.mpg.de/Stiefel-Whitney_characteristic_classes}}

\newcommand{\sz}{\bibitem[SZ05]{SZ} \emph{T. Sch\"oneborn and G. Ziegler}, The Topological Tverberg Theorem and Winding Numbers, J. Comb. Theory, Ser. A, 112:1 (2005) 82--104, arXiv:math/0409081.}

\newcommand{\szno}{\bibitem[Sz91]{Sz91} \emph{A.~Sz\"ucs,} On the cobordism groups of immersions and embeddings,
Math. Proc. Camb. Phil. Soc., 109 (1991) 343--349.}


\newcommand{\ta}{\bibitem[Ta]{Ta} * Handbook of Graph Drawing and Visualization. ed. by R. Tamassia, CRC Press, 2016.}


\newcommand{\tazz}{\bibitem[Ta00]{Ta00} \emph{K. Taniyama,} Higher dimensional links in a simplicial complex embedded in a sphere, Pacific Jour. of Math. 194:2 (2000), 465-467.}

\newcommand{\theo}{\bibitem[Th81]{Th81} * \emph{C.~Thomassen,} Kuratowski's theorem, J.~Graph. Theory 5 (1981), 225--242.}

\newcommand{\tooo}{\bibitem[To11]{To11} \emph{Tonkonog D.} Embedding 3-manifolds with boundary into closed 3-manifolds, Topol. Appl. 158 (2011), 1157-1162. arXiv:1003.3029.}


\newcommand{\tsbzf}{\bibitem[TSB]{TSB} \emph{D. M. Thilikos, M. Serna and H. L. Bodlaender},
Cutwidth I: A linear time fixed parameter algorithm, J. of Algorithms, 56:1 (2005), 1--24.}


\newcommand{\tsbzfd}{\bibitem[TSB05']{TSB05'} \emph{D. M. Thilikos, M. Serna and H. L. Bodlaender},
Cutwidth II: , J. of Algorithms, 56:1 (2005), 25--49.}



\newcommand{\umse}{\bibitem[Um78]{Um78} \emph{B. Ummel.} The product of nonplanar complexes does not imbed in 4-space, Trans. Amer. Math. Soc., 242 (1978) 319--328.}




\newcommand{\val}{\bibitem[Val]{Val} * \url{https://en.wikipedia.org/wiki/Valknut}}


\newcommand{\vi}{\bibitem[Vi]{Vi} * \emph{O. Viro.}
Some integral calculus based on Euler characteristic, Lect. Notes in Math. 1346.}

\newcommand{\vizt}{\bibitem[Vi02]{Vi02} * \emph{Э. Б. Винберг.} Курс алгебры. Москва. Факториал Пресс. 2002.}

\newcommand{\vizteng}{\bibitem[Vi02]{Vi02} * \emph{E. B. Vinberg.} A Course in Algebra. Graduate Studies in Mathematics, vol. 56. 2003.}

\newcommand{\vinhzs}{\bibitem[VINH07]{VINH07} * \emph{О. Я. Виро, О. А. Иванов, Н. Ю. Нецветаев и В. М. Харламов.}
Элементарная топология, МЦНМО. 2007.}

\newcommand{\vktt}{\bibitem[vK32]{vK32} \emph{E.~R.~van~Kampen}, Komplexe in euklidischen R\"aumen, Abh. Math. Sem. Hamburg, 9 (1933) 72--78; Berichtigung dazu, 152--153.
English translation by Tu T$\hat a$m Ngu$\tilde{\hat e}$n-Phan:
\url{https://sites.google.com/site/tutamnguyenphan/van_Kampen.pdf}}

\newcommand{\kafo}{\bibitem[vK41]{vK41} \emph{E. R. van Kampen,} Remark on the address of S. S. Cairns,
in Lectures in Topology, 311--313, University of Michigan Press, Ann Arbor, MI, 1941.}

\newcommand{\vo}{\bibitem[Vo96]{vo96} \emph{A. Yu. Volovikov,} On a topological generalization of the Tverberg theorem. Math. Notes 59:3 (1996), 324--326.}

\newcommand{\vopns}{\bibitem[Vo96v]{Vo96v} \emph{A. Yu. Volovikov,} On the van Kampen-Flores Theorem.
Math. Notes 59:5 (1996), 477--481.}

\newcommand{\vznt}{\bibitem[VZ93]{VZ93} \emph{A. Vu\v ci\'c and R. T. \v Zivaljevi\'c}, Note on a conjecture of Sierksma, Discr. Comput. Geom. 9 (1993), 339-349.}

\newcommand{\vzzn}{\bibitem[VZ09]{VZ09} \emph{S. T. Vre\'cica and R. T. \v Zivaljevi\'c},  Chessboard complexes
indomitable, J. of Comb. Theory, Ser. A 118:7 (2011), 2157--2166. arXiv:0911.3512.}


\newcommand{\walst}{\bibitem[Wa62]{Wa62} \emph{C.~T.~C.~Wall}, Classification of $(n-1)$-connected $2n$-manifolds, Ann. of Math., 75 (1962) 163--189.}


\newcommand{\wallss}{\bibitem[Wa67]{Wa67} \emph{C.~T.~C.~Wall.} Classification problems in differential topology, IV, Thickenings, Topology 1966. 5. P. 73--94.}

\newcommand{\waldss}{\bibitem[Wa67m]{Wa67m} \emph{F. Waldhausen.} Eine Klasse von 3-dimensional Mannigfaltigkeiten, I. Invent. Math. 1967. 3. P.~308-333.}

\newcommand{\walsz}{\bibitem[Wa70]{Wa70} \emph{C. T. C. Wall,} Surgery on compact manifolds,
1970, Academic Press, London.}

\newcommand{\wess}{\bibitem[We67]{We67} \emph{C.~Weber.} Plongements de poly\`edres dans le domain metastable, Comment. Math. Helv. 42 (1967), 1--27.}

\newcommand{\whit}{\bibitem[Wl]{Wl} * \url{https://en.wikipedia.org/wiki/Whitehead_link}}

\newcommand{\winum}{\bibitem[Wn]{Wn} * \url{https://en.wikipedia.org/wiki/Winding_number}}

\newcommand{\wrss}{\bibitem[Wr77]{Wr77} \emph{P. Wright.} Covering 2-dimensional polyhedra by 3-manifolds spines.
Topology. 16 (1977), 435--439.}

\newcommand{\wufe}{\bibitem[Wu58]{Wu58} \emph{W. T. Wu.} On the realization of complexes in a euclidean space (in Chinese): I, Sci Sinica, 7 (1958) 251--297; II, Sci Sinica, 7 (1958) 365--387; III, Sci Sinica, 8 (1959) 133--150.}

 \newcommand{\wufn}{\bibitem[Wu59]{Wu59} \emph{W.~T.~Wu.} On the isotopy of a finite complex in Euclidean space, I, II, Science Record, N.S. 3:8 (1959) 342--347, 348--351.}

\newcommand{\wusf}{\bibitem[Wu65]{Wu65} * \emph{W. T. Wu.} A Theory of Embedding, Immersion and Isotopy of Polytopes in an Euclidean Space. Peking: Science Press, 1965.}


\newcommand{\yann}{\bibitem[Ya99]{Ya99} \emph{Z. Yang.} Computing Equilibria and Fixed Points: The Solution of Nonlinear Inequalities, Kluwer, Springer Science + Business Media, 1990.}


\newcommand{\zesz}{\bibitem[Ze60]{Ze60} \emph{E. C. Zeeman}, Unknotting spheres in five dimensions, Bull. Amer. Math. Soc. 66 (1960) 198.
\linebreak
\url{https://www.ams.org/journals/bull/1960-66-03/S0002-9904-1960-10431-4/S0002-9904-1960-10431-4.pdf}}

\newcommand{\z}{\bibitem[Ze]{Z} * \emph{E. C. Zeeman}, A Brief History of Topology, UC Berkeley, October 27, 1993, On the occasion of Moe Hirsch's 60th birthday, \url{http://zakuski.utsa.edu/~gokhman/ecz/hirsch60.pdf}.}

\newcommand{\zioz}{\bibitem[Zi10]{Zi10} * \emph{D. \v Zivaljevi\'c}, Borromean and Brunnian Rings
\url{http://www.rade-zivaljevic.appspot.com/borromean.html}.}

\newcommand{\zioo}{\bibitem[Zi11]{Zi11} * \emph{G. M. Ziegler}, 3N Colored Points in a Plane, Notices of the Amer. Math. Soc., 58:4 (2011), 550-557.}


\newcommand{\zot}{\bibitem[Zi13]{Z13} \emph{A. Zimin.} Alternative proofs of the Conway-Gordon-Sachs Theorems, arXiv:1311.2882.}


\newcommand{\zss}{\bibitem[ZSS]{ZSS} * Элементы математики в задачах: через олимпиады и кружки к профессии
Сборник под редакцией А. Заславского, А. Скопенкова и М. Скопенкова. Изд-во МЦНМО, 2018.
\url{http://www.mccme.ru/circles/oim/materials/sturm.pdf}.}


\newcommand{\zu}{\bibitem[Zu]{Zu} \emph{J. Zung.} A non-general-position Parity Lemma,
\url{http://www.turgor.ru/lktg/2013/1/parity.pdf}.}







\cget

\fwztz

\gatt


\jonly{\kstz}

\lsne

\mazt
\mtwoo

\rsnn

\saeo
\skde
\skne
\skzt
\skze
\sks
\skos
\skosh
\sktzr
\ssnt
\sssne
\stez
\stos

\tazz

\wess

\end{thebibliography}
\end{document}